\documentclass[12pt,a4paper]{article}
\usepackage[utf8]{inputenc}
\usepackage{amsmath}
\usepackage{amsfonts}
\usepackage{indentfirst}
\usepackage{enumerate}
\usepackage{enumitem}
\usepackage{amssymb}
\usepackage{graphicx}
\usepackage{amsthm}
\usepackage{hyphenat}
\usepackage{verbatim}
\usepackage{bold-extra}
\usepackage{changepage}

\author{Luke Warren}
\title{On slow escaping and non-escaping points of quasimeromorphic mappings}
\date{}

\begin{document}

\maketitle
\newtheorem{thm}{Theorem}[section]
\newtheorem{lem}[thm]{Lemma}
\newtheorem{cor}[thm]{Corollary}
\newtheorem{defn}[thm]{Definition}
\newtheorem{thm*}{Theorem}
\newtheorem{lem*}{Lemma}
\newtheorem{pf*}{Proof of Theorem}\numberwithin{equation}{section}
\theoremstyle{definition}
\newtheorem{ex}[thm]{Example}
\newtheorem{ack*}{Acknowledgements}

\newcommand{\capacity}{\operatorname{cap}}
\newcommand{\card}{\operatorname{card}}
\newcommand{\dist}{\operatorname{dist}}
\newcommand{\real}{\operatorname{Re}}
\newcommand{\im}{\operatorname{Im}}

	\begin{adjustwidth}{1cm}{1cm}
		\textbf{Abstract.} We show that for any quasimeromorphic mapping with an essential singularity at infinity, there exist points whose iterates tend to infinity arbitrarily slowly. This extends a result by Nicks for quasiregular mappings, and Rippon and Stallard for transcendental meromorphic functions on the complex plane. We further establish a new result for the growth rate of quasiregular mappings near an essential singularity, and briefly extend some results regarding the bounded orbit set and the bungee set to the quasimeromorphic setting.
	\end{adjustwidth}

\section{Introduction}

First introduced and studied by Eremenko \cite{Eremenko} for transcendental entire functions, and later extended to transcendental meromorphic functions $f$ by Dom\'{i}nguez \cite{Dominguez}, the escaping set is defined as
\begin{equation*}
I(f) = \{z \in \mathbb{C}: f^n(z) \neq \infty \text{ for all } n \in \mathbb{N}, f^n(z) \to \infty \text{ as } n \to \infty\}.
\end{equation*}
It has been shown in \cite{Dominguez, Eremenko} that $I(f) \neq \varnothing$ and the escaping set is strongly related to the Julia set, via $J(f) \cap I(f) \neq \varnothing$ and $J(f) = \partial I(f)$. Since then, properties of the escaping set have been extensively studied; see for example \cite{Marti-Pete1, RS3, RS4, RRRS1, Sixsmith1}

The fast escaping set $A(f) \subset I(f)$ was introduced by Bergweiler and Hinkkanen \cite{BH1} for transcendental entire functions. Subsequently, it was asked whether all escaping points could be fast escaping. Rippon and Stallard \cite{RS5} proved that this is not the case even for transcendental meromorphic mappings, showing that there always exist points in $J(f)$ that escape arbitrarily slowly under iteration. Other results in complex dynamics surrounding slow escape and different rates of escape have been studied in \cite{BP1, RS1, Sixsmith2, Waterman1}.

Quasiregular mappings and quasimeromorphic mappings generalise analytic and meromorphic functions on the plane to higher-dimensional Euclidean space $\mathbb{R}^d$, $d \geq 2$, respectively. We say that a quasiregular or quasimeromorphic mapping on $\mathbb{R}^d$ is of transcendental type if it has an essential singularity at infinity. In this new setting, some analogous results for the escaping set also hold; see \cite{BDF, BFLM}. In particular, Nicks \cite{Nicks1} recently extended the slow escape result to the case of quasiregular mappings of transcendental type. We defer the definition of quasiregular and quasimeromorphic mappings until Section~2.

Recently, the Julia set has been investigated for quasimeromorphic mappings of transcendental type with at least one pole in \cite{Warren1}, as follows.
\begin{align} \label{JuliaSetDefn}
J(f):= \{ &x \in \hat{\mathbb{R}}^d \setminus \overline{\mathcal{O}^{-}_{f}(\infty)} : \card(\hat{\mathbb{R}}^d \setminus \mathcal{O}^{+}_{f}(U_{x}))< \infty \text{ for all} \nonumber \\ 
&\text{neighbourhoods } U_{x} \subset \hat{\mathbb{R}}^d \setminus \overline{\mathcal{O}^{-}_{f}(\infty)} \text{ of } x \} \cup \overline{\mathcal{O}^{-}_{f}(\infty)}.
\end{align}
Here for $x \in \hat{\mathbb{R}}^d = \mathbb{R}^d \cup \{\infty\}$, we use the notation $\mathcal{O}^{-}_{f}(x)$ to denote the backward orbit of $x$, while we use $\mathcal{O}^{+}_{f}(x)$ to denote the forward orbit of $x$.

Using similar techniques to those from \cite{Nicks1} and \cite{BFLM}, it has been possible to extend the slow escape result to the case of quasimeromorphic mappings of transcendental type with at least one pole.

\begin{thm} \label{MainTheorem}
	Let $f: \mathbb{R}^d \to \hat{\mathbb{R}}^d$ be a quasimeromorphic map of transcendental type with at least one pole. Then for any positive sequence $a_n \to \infty$, there exists $\zeta \in J(f)$ and $N \in \mathbb{N}$ such that $|f^n(\zeta)| \to \infty$ as $n \to \infty$, while also $|f^n(\zeta)| \leq a_n$ whenever $n \geq N$.
\end{thm}	
	
	Although Rippon and Stallard \cite{RS5} proved this theorem for transcendental meromorphic functions, their method relied on results that do not extend to the quasimeromorphic setting. In particular, in the case when there are infinitely many poles, they used a version of the Ahlfors five islands theorem. The proof given here offers an alternative proof in the meromorphic case which is, in some sense, more elementary.

	During Nicks's proof of the existence of slow escaping points in \cite{Nicks1}, an important growth result by Bergweiler was needed \cite[Lemma~3.3]{Bergweiler2}, which was concerned with the growth rate of quasiregular mappings of transcendental type defined on the whole of $\mathbb{R}^d$. We have been able to extend this result to the case where the mapping is quasiregular in a neighbourhood of an essential singularity. 
	
	In what follows, we denote the region between two spheres centered at the origin of radii $0 \leq r < s \leq \infty$, by
\begin{equation*}
A(r,s) = \{ x \in \mathbb{R}^d : r < |x| < s \}.
\end{equation*}
Further, for a quasiregular mapping $f: A(R,S) \to \mathbb{R}^d$ and a given $R<r<S$, the maximum modulus is defined by $M(r,f) = \max\{|f(x)| : |x|=r\}$.

\begin{thm} \label{GrowthLimitTheorem}
	Let $R > 0$, let $f: A(R, \infty) \to \mathbb{R}^d$ be a quasiregular map with an essential singularity at infinity, and let $A>1$. Then 
	\begin{equation*}
	\lim_{r \to \infty} \frac{M(Ar,f)}{M(r,f)} = \infty.
	\end{equation*}
\end{thm}

	As an immediate consequence of Theorem~\ref{GrowthLimitTheorem}, we get the following useful corollary.
	
\begin{cor} \label{GrowthLimitTheoremCorollary}
	Let $R > 0$ and $f: A(R, \infty) \to \mathbb{R}^d$ be a quasiregular map with an essential singularity at infinity. Then 
	\begin{equation*}
	\lim_{r \to \infty} \frac{\log M(r,f)}{\log r} = \infty.
	\end{equation*}
\end{cor}

Theorem~\ref{GrowthLimitTheorem} and Corollary~\ref{GrowthLimitTheoremCorollary} will be used in the proof of Theorem~\ref{MainTheorem} in the case when there are finitely many poles. Furthermore, Theorem~\ref{GrowthLimitTheorem} can be applied in the proof of \cite[Lemma~2.6]{NS1} to rectify an omission there. Namely, in \cite{NS1} it is claimed that the proof of a statement like Theorem~\ref{GrowthLimitTheorem} is similar to the proof of Bergweiler's result \cite[Lemma~3.3]{Bergweiler2}. However, part of the proof in \cite{Bergweiler2} relies upon the function being quasiregular on the whole of $\mathbb{R}^d$. This means that it cannot be applied when the function is only quasiregular in a neighbourhood of an essential singularity, as in both \cite[Lemma~2.6]{NS1} and Theorem~\ref{GrowthLimitTheorem}. Nonetheless, we will show in Section~3 that it is possible to significantly adapt the ideas in \cite{Bergweiler2} to obtain a proof of Theorem~\ref{GrowthLimitTheorem}. These new results may be of independent interest.

	Alongside the escaping set $I(f)$, it is useful to consider the sets 
	\begin{align*}
	BO(f) &:=\{x \in \mathbb{R}^d : \{f^n(x): n \in \mathbb{N}\} \text{ is bounded}\}, \text{ and}\\
	BU(f) &:= \mathbb{R}^d \setminus \left(I(f) \cup BO(f) \cup \mathcal{O}_{f}^{-}(\infty) \right).
	\end{align*}
These sets are known as the bounded orbit set and the bungee set respectively; $BO(f)$ consists of points with a bounded forward orbit, while $BU(f)$ consists of points $x$ whose sequence of iterates $(f^n(x))$ contains both a bounded subsequence and a subsequence that tends to infinity. Together with $I(f)$ and $\mathcal{O}_{f}^{-}(\infty)$, these sets partition $\mathbb{R}^d$ based on the behaviour of the forward orbit of the points. Further, it is clear by their definitions that $BO(f)$ and $BU(f)$ are also completely invariant under $f$. 

For a transcendental entire function $f$, the sets $BO(f)$ and $BU(f)$ have been well studied; for the former, see for example \cite{Bergweiler5, Osborne1}, while for the latter we refer to \cite{EL2, OS1, Sixsmith3}. It should be noted that $BO(f)$ is often denoted as $K(f)$ in the literature, however this notation is not used in the quasiregular setting because $K(f)$ is reserved for the dilatation of a quasiregular mapping.

When $f$ is a transcendental meromorphic function, by following a similar argument to that given in \cite[Proof of Theorem~1.1]{OS1} and using the fact that $BU(f) \neq \varnothing$ (which shall follow from Theorem~\ref{BoundaryTheorem}), we get the following relationship between these sets and the Julia set.

\begin{equation}\label{BoundaryEquality}
J(f) = \partial I(f) = \partial BO(f) = \partial BU(f).
\end{equation}

Some results for $BO(f)$ and $BU(f)$ were successfully extended to the case where $f$ is quasiregular of transcendental type in \cite{BN1} and \cite{NS3} respectively. For instance, it was shown that both $BO(f)$ and $BU(f)$ intersect $J(f)$ infinitely often, and $J(f) \subset \partial I(f) \cap \partial BO(f)$. Further, for many quasiregular mappings of transcendental type we also have that $J(f) \subset \partial BU(f)$. However, examples in \cite{BN1, NS3} show that \eqref{BoundaryEquality} does not extend to entire quasiregular mappings of transcendental type. 

For quasimeromorphic mappings of transcendental type with at least one pole, we find that analogous results hold. 

\begin{thm}\label{BoundaryTheorem}
Let $f: \mathbb{R}^d \to \hat{\mathbb{R}}^d$ be a quasimeromorphic map of transcendental type with at least one pole. Then 
\begin{enumerate}
\item[\emph{(i)}] $BO(f) \cap J(f)$ and $BU(f) \cap J(f)$ are infinite, and
\item[\emph{(ii)}] $J(f) \subset \partial I(f) \cap \partial BO(f) \cap \partial BU(f)$.
\end{enumerate}
\end{thm}
 
	By extending the examples mentioned above, we can show that equality in Theorem~\ref{BoundaryTheorem}(ii) need not hold for general mappings in the new setting. Example~\ref{Example1} and Example~\ref{Example2} shall show that it is possible to have $(\partial I(f) \cap \partial BO(f)) \setminus J(f) \neq \varnothing$ and $\partial BU(f) \setminus J(f) \neq \varnothing$ respectively.

	The majority of this paper will be dedicated to the proof of Theorem~\ref{MainTheorem}, which shall be completed in two parts. Firstly, Section~2 will be dedicated to stating definitions and preliminary results. In Section~3 we will prove Theorem~\ref{GrowthLimitTheorem}. From here, following a similar argument by Nicks \cite{Nicks1}, the case when the mapping $f$ has finitely many poles in Theorem~\ref{MainTheorem} will be proven in Section~4, by considering whether $f$ has the `pits effect' (see Section~4.1) or not. In Section~5, we treat the remaining case where $f$ has infinitely many poles. Finally, in Section~6 we will prove Theorem~\ref{BoundaryTheorem} and provide counterexamples to \eqref{BoundaryEquality} in the new setting.
	
\section{Preliminary results}
	
	\subsection{Quasiregular and quasimeromorphic mappings}
	
For notation, for $d \geq 2$ and $x \in \mathbb{R}^d$ we denote the $d$-dimensional ball centered at $x$ of radius $r>0$ as $B(x,r) = \{y \in \mathbb{R}^d : |x-y|<r\}$. We also denote the $(d-1)$-sphere centered at the origin of radius $r>0$ by $S(r) = \partial B(0,r)$.	
	
	We shall briefly recall the definition and some main results of quasiregular and quasimeromorphic mappings here. For a more comprehensive introduction to these mappings, we refer to \cite{MRV1}, \cite{Reshetnyak1} and \cite{Rickman3}.
	
Let $d \geq 2$ and $U \subset \mathbb{R}^d$ be a domain. For $1 \leq p < \infty$, the Sobolev space $W_{p,loc}^{1}(U)$ consists of all functions $f:U \to \mathbb{R}^d$ for which all first order weak partial derivatives exist and are locally in $L^p(U)$. A continuous map $f \in  W_{d,loc}^{1}(U)$ is called quasiregular if there exists some constant $K\geq 1$ such that 
\begin{equation}\label{qrDefn1}
\left(\sup_{|h|=1}|Df(x)(h)|\right)^d \leq KJ_{f}(x) \text{ a.e.},
\end{equation}
where $Df(x)$ denotes the derivative of $f(x)$ and $J_{f}(x)$ denotes the Jacobian determinant. The smallest constant $K$ for which \eqref{qrDefn1} holds is called the outer dilatation and denoted $K_{O}(f)$.

If $f$ is quasiregular, then there also exists some $K' \geq 1$ such that
\begin{equation}\label{qrDefn2}
K'\left(\inf_{|h|=1}|Df(x)(h)|\right)^d \geq J_{f}(x) \text{ a.e.}
\end{equation}
The smallest constant $K'$ for which \eqref{qrDefn2} holds is called the inner dilatation and denoted $K_{I}(f)$. Finally, the dilatation of $f$ is defined as $K(f) := \max\{K_{O}(f), K_{I}(f)\}$ and if $K(f) \leq K$ for some $K \geq 1$, then we say that $f$ is $K$-quasiregular.
	
The definition of quasiregularity can be naturally extended to mappings into $\hat{\mathbb{R}}^d := \mathbb{R}^d \cup \{\infty\}$. For a domain $D \subset \mathbb{R}^d$, we say that a continuous map $f : D \to \hat{\mathbb{R}}^d$ is called quasimeromorphic if every $x \in D$ has a neighbourhood $U_{x}$ such that either $f$ or $M \circ f$ is quasiregular from $U_{x}$ into $\mathbb{R}^d$, where $M:\hat{\mathbb{R}}^d \to \hat{\mathbb{R}}^d$ is a sense-preserving M\"{o}bius map such that $M(\infty) \in \mathbb{R}^d$.

	If $f$ and $g$ are quasiregular mappings, with $f$ defined in the range of $g$, then $f \circ g$ is quasiregular, with 
	\begin{equation}\label{compositionDilatation}
	K(f \circ g) \leq K(f)K(g).
	\end{equation}
	
Similarly, if $g$ is a quasiregular mapping and $f$ is a quasimeromorphic mapping defined in the range of $g$, then $f \circ g$ is quasimeromorphic and the above inequality also holds.

It was established by Reshetnyak \cite{Reshetnyak1, Reshetnyak2}, that every non-constant $K$-quasiregular map $f$ is discrete and open. Moreover, many other properties of analytic and meromorphic mappings have analogues for quasiregular and quasimeromorphic mappings, such as the following analogue of Picard's theorem by Rickman \cite{Rickman1, Rickman2}.

\begin{thm} \label{PicardTheorem}
	Let $d \geq 2$, $K \geq 1$. Then there exists a positive integer $\tilde{q_{0}} = \tilde{q_{0}}(d,K)$, called Rickman's constant, such that if $R>0$ and $f : A(R, \infty) \to \hat{\mathbb{R}}^d\setminus\{a_1, a_2, \dots, a_{\tilde{q_{0}}} \}$ is a $K$-quasimeromorphic mapping with $a_{1}, a_{2},\dots,a_{\tilde{q_{0}}} \in \hat{\mathbb{R}}^d$ distinct, then $f$ has a limit at $\infty$.
	
	In particular, if $b_{1}, b_{2},\dots, b_{\tilde{q_{0}}} \in \hat{\mathbb{R}}^d$ are distinct points and $f:\mathbb{R}^d \to \hat{\mathbb{R}}^d$ is a $K$-quasimeromorphic mapping of transcendental type, then there exists some $i \in \{1,2,\dots, \tilde{q_{0}}\}$ such that $f^{-1}(b_{i})$ contains points of arbitrarily large modulus.
\end{thm}

It should be noted that by the above theorem, the exceptional set $E(f) :=\{ x \in \mathbb{R}^d : \mathcal{O}^{-}_{f}(x) \text{ is finite} \}$ has at most $\tilde{q_{0}}$ elements.

For $K$-quasiregular mappings, the quantity $q_{0} := \tilde{q_{0}} -1$ is also referred to as Rickman's constant. This is because infinity is omitted, which is not always the case for $K$-quasimeromorphic mappings. Since the case with finitely many poles reduces down to $K$-quasiregular mappings defined near an essential singularity, we shall mainly use $q_{0}$ and refer to it explicitly as Rickman's quasiregular constant.

Another important theorem is a sufficient condition for when a quasiregular mapping can be extended over isolated points. The following theorem follows from a result first established by Callendar \cite{Callender1}, which was later generalised by Martio, Rickman and V\"{a}is\"{a}l\"{a} \cite{MRV2}.

\begin{thm} \label{thm:removeSing1}
	Let $D \subset \mathbb{R}^d$ be a domain, $E \subset D$ be a finite set of points and $f:D \setminus E \to \mathbb{R}^d$ be a bounded $K$-quasiregular mapping. Then $f$ can be extended to a $K$-quasiregular mapping on all of $D$.
\end{thm}

	\subsection{Capacity of a condenser}
Let $U \subset \mathbb{R}^d$ be an open set and $C \subset U$ be non-empty and compact. We call the pair $(U,C)$ a condenser and define the (conformal) capacity of $(U,C)$, denoted $\capacity(U,C)$, by
\begin{equation*}
\capacity(U,C) := \inf_{\phi}\int_{U}|\nabla \phi|^d dm,
\end{equation*}
where the infimum is taken over all non-negative functions $\phi \in C_{0}^{\infty}(U)$ satisfying $\phi(x) \geq 1$ for all $x \in C$.

It was shown by Reshetnyak \cite{Reshetnyak1} that if $\capacity(U,C)=0$ for some bounded open set $U \supset C$, then $\capacity(V,C)=0$ for all bounded open sets $V \supset U$. In this case, we say that $C$ has zero capacity and write $\capacity(C)=0$; otherwise we say that $C$ has positive capacity and write $\capacity(C)>0$. If $C \subset \mathbb{R}^d$ is an unbounded closed set, then we say that $\capacity(C)=0$ if $\capacity(C')=0$ for every compact set $C' \subset C$. 

It is known from \cite[Theorem 4.1]{Wa2} that $\capacity(C)=0$ implies $C$ has Hausdorff dimension zero. Also, it is known that if $C$ is a countable set, then $\capacity(C)=0$. Hence, we can informally consider sets of capacity zero as `small' sets. 

For a quasimeromorphic mapping of transcendental type with at least one pole, a strong relationship between points with finite backward orbits and capacity was established in \cite{Warren1}.

\begin{thm}\label{EAndCapacity}
Let $f: \mathbb{R}^d \to \hat{\mathbb{R}}^d$ be a quasimeromorphic mapping of transcendental type with at least one pole. Then $x \in E(f)$ if and only if $\capacity\left(\overline{\mathcal{O}_{f}^{-}(x)}\right)=0$. 
\end{thm}

\subsection{Julia set of quasimeromorphic mappings}

The following theorem due to Miniowitz \cite{Miniowitz1} is an extension of Montel's theorem to the quasimeromorphic setting. Here, we denote the chordal distance between two points $x_{1}, x_{2} \in \hat{\mathbb{R}}^d$ by $\chi(x_{1},x_{2})$.

\begin{lem} \label{qmMontel}
Let $\mathcal{F}$ be a family of $K$-quasimeromorphic mappings on a domain $X \subset \mathbb{R}^d, d\geq 2$, and let $\tilde{q_{0}} = \tilde{q_{0}}(d,K)$ be Rickman's constant. 

Suppose that there exists some $\epsilon >0$ such that each $f \in \mathcal{F}$ omits $\tilde{q_{0}}$ values $a_{1}(f), a_{2}(f), \dots, a_{\tilde{q_{0}}}(f) \in \hat{\mathbb{R}}^d$ with $\chi(a_{i}(f), a_{j}(f)) \geq \epsilon$ for all $i \neq j$. Then $\mathcal{F}$ is a normal family on $X$.
\end{lem}

For a general $K$-quasimeromorphic mapping $f$, the dilatation of the iterates $f^k$ can grow exponentially large. As a result, the above theorem cannot be applied to the family of iterates to study the Julia set in this case. Nonetheless, it can be applied to a re-scaled family of mappings $\left\lbrace f(rx)/s : r,s \in \mathbb{R} \right\rbrace$, since all members of this family have the same dilatation $K$.

By defining the Julia set directly using the expansion property in \eqref{JuliaSetDefn}, it has been possible to study analogues of the Fatou-Julia theory in the new setting. Recently, the Julia set for quasimeromorphic mappings of transcendental type with at least one pole has been successfully established in \cite{Warren1}; here, it was shown that many of the usual properties of the Julia set analogously hold as well. These are summarised below.

\begin{thm}\label{JuliaSetProperties}
Let $f:\mathbb{R}^d \to \hat{\mathbb{R}}^d$ be a quasimeromorphic mapping of transcendental type with at least one pole. Then the following hold.

\begin{enumerate}
\item[\emph{(i)}] $J(f) \neq \varnothing$. In fact, $\card(J(f)) = \infty$.
\item[\emph{(ii)}] $J(f)$ is perfect.
\item[\emph{(iii)}] $x \in J(f) \setminus \{ \infty \}$ if and only if $f(x) \in J(f)$. In particular, $J(f) \setminus \mathcal{O}^{-}_{f}(\infty)$ is completely invariant.
\item[\emph{(iv)}] $J(f) \subset \overline{\mathcal{O}^{-}_{f}(x)}$ for every $x \in \hat{\mathbb{R}}^d \setminus E(f)$.
\item[\emph{(v)}] $J(f) = \overline{\mathcal{O}^{-}_{f}(x)}$ for every $x \in J(f) \setminus E(f)$.
\item[\emph{(vi)}] Let $U \subset \hat{\mathbb{R}}^d$ be an open set such that $U \cap J(f) \neq \varnothing$. Then for all $x \in \hat{\mathbb{R}}^d \setminus E(f)$, there exists some $w \in U$ and some $k \in \mathbb{N}$ such that $f^k(w)=x$.
\item[\emph{(vii)}] For each $n \in \mathbb{N}$,
\begin{align*}
J(f)= \{ &x \in \hat{\mathbb{R}}^d \setminus \overline{\mathcal{O}^{-}_{f}(\infty)} : \card(\hat{\mathbb{R}}^d \setminus \mathcal{O}^{+}_{f^n}(U_{x}))< \infty \text{ for all} \nonumber \\ 
&\text{neighbourhoods } U_{x} \subset \hat{\mathbb{R}}^d \setminus \overline{\mathcal{O}^{-}_{f}(\infty)} \text{ of } x \} \cup \overline{\mathcal{O}^{-}_{f}(\infty)}.
\end{align*}
\end{enumerate}
\end{thm}

We remark that the Julia set definition in \eqref{JuliaSetDefn} is different to the Julia set definition used for quasiregular mappings of transcendental type, which were defined by Bergweiler and Nicks in \cite{BN1}. For those mappings, the cardinality condition is replaced by a weaker condition using conformal capacity. Although these conditions are equivalent for quasimeromorphic mappings of transcendental type with at least one pole by Theorem~\ref{EAndCapacity}, it remains an open conjecture whether this result can be extended to quasiregular mappings of transcendental type; see \cite{BN1}. For this reason, we include the extra condition that each quasimeromorphic mapping has at least one pole in the statement of the theorems within this paper.

\subsection{Brouwer degree and covering lemmas}

Let $f: G \to \mathbb{R}^d$ be a quasiregular mapping, $D \subset G$ be an open set with $\overline{D} \subset G$ compact, and let $y \in \mathbb{R}^d \setminus f(\partial D)$. Firstly, for $x \in G$, we define the local (topological) index of $f$ at $x$, denoted by $i(x,f)$, as
\begin{equation*}
i(x,f) := \inf\{\sup\{\card(f^{-1}(w) \cap U_{x}) : w \in \mathbb{R}^d \}\},
\end{equation*}
where the infimum is taken over all the neighbourhoods $U_{x} \subset G$ of $x$. 

From here we define the Brouwer degree of $f$ at $y$ over $D$, denoted $\mu(y,f,D)$, as

\begin{equation} \label{muDefn}
	\mu(y,f,D) = \sum_{x \in f^{-1}(y) \cap D} i(x,f),
	\end{equation}
which informally counts the number of preimages of $y$ in $D$ including multiplicity.

For quasiregular mappings, the Brouwer degree has many useful properties which will be summarised below without proof (See \cite[Section~II.2.3]{RR} and \cite[Proposition~I.4.4]{Rickman3}).

\begin{thm} \label{DegreeProperties} 
	Let $f: G \to \mathbb{R}^d$ be a quasiregular mapping and let $D \subset \mathbb{R}^d$ be an open bounded set with $\overline{D} \subset G$. Then the following hold: 
	\begin{enumerate}
	\item[\emph{(i)}] 
	 if $x, y \not\in f(\partial D)$ are in the same connected component of $\mathbb{R}^d \setminus f(\partial D)$, then $\mu(x,f,D) = \mu(y,f,D)$.
	\item[\emph{(ii)}] 
	if $y \not\in f(\partial D)$, $X_{1}, X_{2}, \dots, X_{n}$ are disjoint sets and if $D \cap f^{-1}(y) \subset \bigcup_{i} X_{i} \subset D$, then 
	\begin{equation*}
	\mu(y,f,D) = \sum_{i=1}^{n} \mu(y,f,X_{i}) \text{ (if defined)}.
	\end{equation*}
	\item[\emph{(iii)}] 
	 if $y \not\in f(\partial D)$ and $g: H \to \mathbb{R}^d$ is a quasiregular mapping with $\overline{D} \subset H$ such that $\max\{|f(x)-g(x)| : x \in \partial D \} < \min \{|f(x)-y| : x \in \partial D\}$, then $\mu(y,f,D) = \mu(y,g,D)$.
	\item[\emph{(iv)}] 
	if $\alpha, \beta >0$ and $\alpha y \not\in f(\partial D)$,  then 
	\begin{equation*}
	\mu(\alpha y,f,D)= \mu(y, F,D'),
	\end{equation*} 
	where $D' = (1/\beta)D$ and $F: \Omega \to \mathbb{R}^d$ is a quasiregular mapping with $\Omega \supset \overline{D'}$, defined by $F(x) = (1/\alpha)f(\beta x)$.
	\end{enumerate}
\end{thm}

The following covering lemma is an extension of \cite[Lemma~3.1]{Sixsmith2} to the quasimeromorphic setting.

\begin{lem} \label{NestedSeqn}
	Let $f:\mathbb{R}^d \to \hat{\mathbb{R}}^d$ be a continuous function. For $n \geq 0$, let $(F_{n})$ be a sequence of non-empty bounded sets in $\mathbb{R}^d$, $(\ell_{n+1})$ be a sequence of natural numbers and $G_{n} \subset F_{n}$ be a sequence of non-empty subsets such that $f^{\ell_{n+1}}$ is continuous on $\overline{G_{n}}$ with
	\begin{equation}\label{Containment}
	f^{\ell_{n+1}}(G_{n}) \supset F_{n+1}.
	\end{equation} 
	For $n \in \mathbb{N}$, set $r_{n} = \sum_{i=1}^{n} \ell_{i}$. Then there exists $\zeta \in \overline{F_{0}}$ such that $f^{r_{n}}(\zeta) \in \overline{F_{n}}$ for each $n \in \mathbb{N}$.
	
	Further, suppose that $f:\mathbb{R}^d \to \hat{\mathbb{R}}^d$ is a quasimeromorphic mapping of transcendental type with at least one pole such that for $n \geq 0$, $f^{\ell_{n+1}}$ is quasimeromorphic on $\overline{G_{n}}$ and \eqref{Containment} holds. If there is a subsequence $(F_{n_{k}})$ such that $\overline{F_{n_{k}}} \cap J(f) \neq \varnothing$ for all $k \in \mathbb{N}$, then $\zeta$ can be chosen to be in $J(f) \cap \overline{F_{0}}$.
	\end{lem}

\begin{proof}
For all $n \geq 0$, $f^{\ell_{n+1}}$ is continuous on $\overline{G_{n}}$ and $\overline{G_{n}}$ is compact, so \eqref{Containment} implies that $f^{\ell_{n+1}}(\overline{G_{n}}) \supset \overline{F_{n+1}}$ for all $n \geq 0$. Now define the sets 
\begin{equation*}
T_{N} = \{ x \in \overline{G_{0}} : f^{r_{n}}(x) \in \overline{G_{n}} \text{ for all }  n \leq N\}.
\end{equation*}

The sets $T_{N}$ are non-empty, compact and form a decreasing nested sequence. Thus $T:= \bigcap_{N=1}^{\infty}T_{N}$ is non-empty and any $\zeta \in T$ is such that $f^{r_{n}}(\zeta) \in \overline{F_{n}}$ for all $n \in \mathbb{N}$.

Now suppose that $f$ is a quasimeromorphic mapping of transcendental type with at least one pole and for $n \geq 0$, $f^{\ell_{n+1}}$ is quasimeromorphic on $\overline{G_{n}}$ and \eqref{Containment} holds. Since $J(f)$ is backward invariant, we get that $\overline{G_{n}} \cap J(f) \neq \varnothing$ for all $n \geq 0$. It follows that $f^{\ell_{n+1}}(\overline{G_{n}} \cap J(f)) \supset \overline{F_{n+1}} \cap J(f)$ for all $n \geq 0$. 

By applying the first part of the lemma to the closed sets $\overline{F_{n}} \cap J(f)$, then $\zeta \in J(f) \cap \overline{F_{0}}$ as required. 
\end{proof}

It should be noted that by setting $\ell_{n} = 1$ for all $n \in \mathbb{N}$, we get a modified version of \cite[Lemma~1]{RS5}. This version shall be used for the proof of Theorem~\ref{MainTheorem}, while the general version shall be reserved for the proof of Theorem~\ref{BoundaryTheorem}.

\subsection{Holding-up lemma}

For a quasimeromorphic mapping with finitely many poles, it is possible to get sufficient conditions for the existence of a slow escaping point using the same `holding-up' technique as that for quasiregular mappings of transcendental type. The proof of the following lemma is similar to that by Nicks \cite[Lemma~3.1]{Nicks1} and is therefore omitted. 

\begin{lem} \label{lem:HoldUp}
	Let $f:\mathbb{R}^d \to \hat{\mathbb{R}}^d$ be a $K$-quasimeromorphic function of transcendental type with at least one pole. Let $p \in \mathbb{N}$ and, for $m \in \mathbb{N}$ and $i \in \{1,2,\dots, p\}$, let $X_{m}^{(i)} \subset \mathbb{R}^d$ be non-empty bounded sets, with $X_{m} = \bigcup_{i=1}^{p} X_{m}^{(i)}$, such that 
	\begin{equation} \label{SmallestXToInfty}
	\inf\{|x| : x \in X_{m} \} \to \infty \text{ as } m \to \infty.
	\end{equation}
	
	Suppose further that
	\begin{enumerate}
	\item [\emph{(X1)}] for all $m \in \mathbb{N}$ and $i \in \{1,2,\dots, p\}$, there exists some $j \in \{1,2,\dots, p\}$ such that $f\left(X_{m}^{(i)} \right) \supset X_{m+1}^{(j)}$,
	\end{enumerate}
	and there exists a strictly increasing sequence of integers $(m_{t})$ such that 
	\begin{enumerate}
	\item [\emph{(X2)}] for all $t \in \mathbb{N}$ and $i \in \{1,2,\dots, p\}$, there exists some $j \in \{1,2,\dots, p\}$ such that $f\left(X_{m_{t}}^{(i)} \right) \supset X_{m_{t}}^{(j)}$, and 
	\item [\emph{(X3)}] for all $t \in \mathbb{N}$ and $i \in \{1,2,\dots, p\}$, $\overline{X_{m_{t}}^{(i)}} \cap J(f) \neq \varnothing$.
	\end{enumerate}
	
	Then given any positive sequence $a_n \to \infty$, there exists $\zeta \in J(f)$ and $N_{1} \in \mathbb{N}$ such that $|f^n(\zeta)| \to \infty$ as $n \to \infty$, while also $|f^n(\zeta)| \leq a_n$ whenever $n \geq N_{1}$.
\end{lem}

\section{Growth result for quasiregular mappings near an essential singularity}

Before we begin the proof of Theorem~\ref{GrowthLimitTheorem}, we will first note the following fact about the maximum modulus for quasiregular mappings defined in a neighbourhood of an essential singularity; this follows from the maximum modulus principle and an application of Theorem~\ref{PicardTheorem}.

\begin{lem} \label{R'lemma}
	Let $R > 0$ and let $f: A(R, \infty) \to \mathbb{R}^d$ be a $K$-quasiregular mapping with an essential singularity at infinity. Then there exists $R' \geq R$ such that $M(r,f)$ is a strictly increasing function for $r\geq R'$.
\end{lem}

	Using the above, we now aim to prove Theorem~\ref{GrowthLimitTheorem}. We will assume without loss of generality that $R>0$ is sufficiently large such that $f: A(R,\infty) \to \mathbb{R}^d$ is a $K$-quasiregular mapping with an essential singularity at infinity and $M(r,f)$ is a strictly increasing function for $r\geq R$.
	
	Now let $A>1$ be given and suppose for a contradiction to Theorem~\ref{GrowthLimitTheorem} that there exists some constant $L>1$ and some real sequence $r_{n} \to \infty$ such that $M(Ar_{n},f) \leq LM(r_{n},f)$. By taking subsequences and then starting from large enough $n$, we may assume that $(r_{n})$ is a strictly increasing sequence with $r_{1} > R$.
	
	Define a new sequence $(f_{n})$ by
	\begin{equation} \label{fnDefn}
	f_{n}(x) := \frac{f(r_{n}x)}{M(r_{n},f)}.
	\end{equation}
	
	For each $N \in \mathbb{N}$, let $A_{N} := A\left(R/r_{N}, A\right)$. Now for all $n \geq N$, $f_{n}$ is well-defined and $K$-quasiregular on $A_{N}$.
	
\begin{lem} \label{NormalLemma} There exists a bounded $K$-quasiregular mapping $h$ defined on $B(0, A) \setminus \{0\}$ and a subsequence of $(f_{n})$ that converges to $h$ locally uniformly on $B(0, A) \setminus \{0\}$.
\end{lem}
	
\begin{proof}
	
	Observe that for each $n \geq N$ and $x \in A_{N}$, 
	\begin{equation} \label{eq:fnBound1}
	|f_{n}(x)| \leq \frac{M(r_{n}|x|,f)}{M(r_{n},f)} \leq \frac{M(Ar_{n},f)}{M(r_{n},f)} \leq L.
	\end{equation}
	
	As $L$ is not dependent on $N$, then $f_{n}$ is uniformly bounded on $A_{N}$ for all $n \geq N$. By Lemma~\ref{qmMontel}, $\mathcal{F}_{N} := \{ f_{n} : n \geq N \}$ is a normal family on $A_{N}$ for each $N \in \mathbb{N}$. In particular, for the sequence $(f_{n}) \subset \mathcal{F}_1$ there exists a subsequence $(f_{1,k})_{k =1}^{\infty} \subset (f_{n})$ such that $(f_{1,k})$ converges locally uniformly on $A_{1}$. Discarding the first term if necessary, we may assume that $(f_{1,k}) \subset \mathcal{F}_2$ so the subsequence is defined and uniformly bounded on $A_{2}$. Thus there exists a subsequence $(f_{2,k})_{k =1}^{\infty} \subset (f_{1,k})$ such that $(f_{2,k})$ converges locally uniformly on $A_{2}$. 
	
	By repeating this process, we build a sequence of subsequences $(f_{1,k}), (f_{2,k}), \dots,$ such that $(f_{i,k}) \supset (f_{i+1,k})$ for all $i \in \mathbb{N}$ and $(f_{i,k})$ converges locally uniformly on $A_{i}$. Now consider the sequence $(f_{k,k})$ and observe that $(f_{k,k})_{k \geq i}$ is a subsequence of each $(f_{i,k})$ with $i \in \mathbb{N}$ by construction. This means that the pointwise limit function
	\begin{equation}\label{hDefn}
	h(w):= \lim_{k \to \infty}f_{k,k}(w)
	\end{equation}
exists on $B(0,A) \setminus \{0\}$.

	Let $D \subset B(0,A) \setminus \{0\}$ be a compact set. Then there exists some $N \in \mathbb{N}$ such that $D \subset A_{N}$ and $(f_{k,k})_{k \geq N}$ is defined on $D$.
	
	Now by construction, $(f_{N,k})$ converges uniformly on $D$. As $(f_{k,k})_{k \geq N}$ is a subsequence of $(f_{N,k})$, then from \eqref{hDefn} we have that $f_{k,k} \to h$ uniformly on $D$. Further, since $(f_{k,k})_{k \geq N}$ is a sequence of $K$-quasiregular mappings on $D$, then $h$ is $K$-quasiregular on $D$. Finally since $D$ was arbitrary, then $f_{k,k} \to h$ locally uniformly on $B(0,A) \setminus \{0\}$ and $h$ is $K$-quasiregular on $B(0,A) \setminus \{0\}$.
	
	
	By discarding terms and relabelling, we may assume that $f_{n} \to h$ locally uniformly on $B(0, A) \setminus \{0\}$. Now by \eqref{eq:fnBound1}, for all $x \in B(0, A) \setminus \{0\}$ we have that $|h(x)| \leq L$, so $h$ is bounded. 
\end{proof}

	By Theorem~\ref{thm:removeSing1}, we can extend $h$ to a $K$-quasiregular mapping defined on $B(0,A)$. By relabelling, let this extended map be $h$. 
	
	Before showing that $h(0)=0$, we make an observation. For each $n \in \mathbb{N}$, let $x_{n} \in S(1)$ be such that $|f(r_{n}x_{n})|=M(r_{n},f)$. As $S(1)$ is compact, then there exists a subsequence $(x_{n_{t}})$ of $(x_{n})$ that converges to some point $\tilde{x} \in S(1)$. Since $f_{n} \to h$ locally uniformly on $B(0,A) \setminus \{0\}$, then it follows that $f_{n_{t}}(x_{n_{t}}) \to h(\tilde{x})$ as $t \to \infty$. Therefore, $|h(\tilde{x})|=1$ for such $\tilde{x} \in S(1)$.

\begin{lem} \label{h0=0}
	Let $h$ be as above. Then $h(0)=0$, so $h$ is a non-constant $K$-quasiregular mapping.
\end{lem}
	
\begin{proof}
	Suppose that $|h(0)| = \zeta \neq 0$. Let $T> 4/\zeta$, $(z_{m}) \subset B(0,A)\setminus \{0\}$ be a sequence such that $z_{m} \to 0$ as $m \to \infty$, and define $C_{m} := S(|z_{m}|)$ for each $m \in \mathbb{N}$. As $h$ is a continuous function, then there exists some $\delta>0$ such that $|h(x) - h(0)|< 1/2T$ whenever $|x| < \delta$. In particular, there exists an $M \in \mathbb{N}$ such that $|z_{m}|<\delta$ whenever $m \geq M$. Hence for all $x \in C_{M}$, we have $|h(x) - h(0)|<1/2T$ whenever $m \geq M$. 
	
	Now as $f_{n} \to h$ locally uniformly on $B(0,A) \setminus \{0\}$ then for all $x \in C_{M}$, there exists some $N_{M} \in \mathbb{N}$ such that $|f_{n}(x) - h(x)|< 1/2T$ whenever $n \geq N_{M}$. Therefore, for every $x \in C_{M}$,
	\begin{equation} \label{eq:smallBall}
	|f_{n}(x) - h(0)| \leq |f_{n}(x) - h(x)| + |h(x) - h(0)|
	< \frac{1}{2T} + \frac{1}{2T} = \frac{1}{T}, 
	\end{equation}	
	whenever $n \geq N_{M}$. Fix such an $n$.
	
	 Since $M(r_{k},f) \to \infty$ as $k \to \infty$, then there exists some $t \in \mathbb{N}$ such that $M(r_{n+t},f)>2M(r_{n},f)$. Now consider $V := A(r_{n}|z_{M}|,r_{n+t}|z_{M}|)$.

	As $n \geq N_{M}$ then from \eqref{eq:smallBall},
	\begin{align*}	 
	f(r_{n}C_{M}) &= M(r_{n},f)f_{n}(C_{M}) \subset B\left(M(r_{n},f)h(0), \frac{M(r_{n},f)}{T}\right) =: B_{n},\text{ and} \\
	f(r_{n+t}C_{M}) &= M(r_{n+t},f)f_{n+t}(C_{M}) \\
	 & \qquad \qquad \qquad \quad \subset B\left(M(r_{n+t},f)h(0), \frac{M(r_{n+t},f)}{T}\right) =: B_{n+t}.
	\end{align*}
Since $M(r_{n+t},f)>2M(r_{n},f)$ and $T \zeta >4$, it follows that $\overline{B}_{n} \cap \overline{B}_{n+t} = \varnothing$.
	
	As $f$ is continuous and open, then $f(V)$ is an open path-connected set. Now there exist $x \in f(V) \cap B_{n}$, $y \in f(V) \cap B_{n+t}$ and a continuous path $\beta:[0,1] \to f(V)$ with endpoints $x$ and $y$. 
	
	Since $\overline{B}_{n}$ and $\overline{B}_{n+t}$ are disjoint, then there must exist some $c \in (0,1)$ such that $\beta(c) \in f(V) \setminus(B_{n} \cup B_{n+t})$. However, as $f$ is open, then $\partial f(V) \subset f(\partial V) \subset B_{n} \cup B_{n+t}$, so $f(V)$ must be unbounded. This contradicts the fact that $f$ is continuous on $\overline{V}$.
\end{proof}
	
	
	Now by Theorem~\ref{PicardTheorem}, there exists some $a \in \mathbb{R}^{d}$ such that $f$ takes the value $a$ infinitely often. Without loss of generality we may assume that $a =0$, else we can consider instead the function $f(x+a)-a$ rather than $f$. We aim to get a contradiction using the Brouwer degree of $f$ and $h$. 
	
	Let $t_{2} \in (0,A)$ be such that $h(x) \neq 0$ for all $x \in S(t_{2})$. Then let $F:= \min \{|h(x)| : x \in S(t_{2}) \} >0$. Since $h(0)=0$ and $h$ is continuous at 0, then we can choose some $t_{1} \in (0, t_{2})$ such that $P:= M(t_{1},h) < F/4$. Now, set $U := A(t_{1}, t_{2})$, so $\overline{U} \subset B(0,A)\setminus \{0\}$. Using this spherical shell, we will show that there exists some point $y$ such that the Brouwer degrees of $f_{n}$ and $h$ at $y$ over $U$ agree for large $n$.

\begin{lem} \label{muConverge}
	Let $f_n$ be defined as in \eqref{fnDefn} and let $h$ be defined as in Lemma~\ref{h0=0}. Then there exists some $N \in \mathbb{N}$ such that whenever $n \geq N$, then $\mu(y,f_n,U) = \mu(y,h,U)$ for some point $y \in f_{n}(U) \cap h(U)$.
\end{lem}
	
\begin{proof}
	As $f_{n} \to h$ uniformly on compact subsets of $B(0,A) \setminus \{0\}$, then there exists $N \in \mathbb{N}$ such that
	\begin{equation} \label{boundaryDistanceBound}
	\sup\{|f_{n}(x)-h(x)| : x \in \partial U \} \leq \sup\{|f_{n}(x) - h(x)| : x \in \overline{U}\} < P, 
	\end{equation}
whenever $n \geq N$. In particular, for all $n \geq N$ and for all $x \in \partial U$, we have $||f_{n}(x)| - |h(x)|| \leq P$. It follows that whenever $n \geq N$, then
	
	\begin{equation} \label{lowerBoundt1}
	M(t_{1},f_{n}) \leq M(t_{1},h) +P = 2P,
	\end{equation}
	and
	\begin{equation} \label{upperBoundt2}
	\min\{|f_{n}(x)| : x \in S(t_{2})\} > \min\{|h(x)| : x \in S(t_{2})\} - \frac{F}{2} = F - \frac{F}{2} = \frac{F}{2}.
	\end{equation} 

Now, for all $n \geq N$ we have that $A\left(2P, F/2 \right) \subset f_{n}(U)$ since the $f_{n}$ are open and continuous. In addition, $A\left(2P, F/2 \right) \subset h(U)$ by construction. Fix some $y \in A\left(2P, F/2 \right)$.

	For all $x \in \partial U$ and $n \geq N$, we have $f_{n}(x) \neq y$ and $h(x) \neq y$. Thus from \eqref{lowerBoundt1} and \eqref{upperBoundt2}, whenever $n \geq N$ we have 
	\begin{align*}
	\min\{ |h(x)-y| : x \in \partial U\} &> \min\{ 2P-M(t_{1},h), \min\{|h(x)| : x \in S(t_{2}) \} - \frac{F}{2} \} \\
	&= \min\{P, \frac{F}{2} \} = P.
	\end{align*}
	Therefore by Theorem~\ref{DegreeProperties}(iii) and \eqref{boundaryDistanceBound}, we conclude that $\mu(y,f_n,U) = \mu(y,h,U)$ whenever $n \geq N$.
\end{proof}

As $h$ is a discrete mapping, then $h^{-1}(y) \cap U$ is a finite set and so 
\begin{equation} \label{dFinite}
	d := \mu(y,h,U) < \infty.
\end{equation}	
Using \eqref{dFinite} and Lemma \ref{muConverge}, we shall now aim for a contradiction by considering the behaviour of $\mu(y,f_n,U)$ as $n \to \infty$.
	
	For $n \geq N$, define $d_{n}= \mu(y,f_{n},U)$, $y_{n} = M(r_{n},f)y$ and $U_{n} = A(r_{n}t_{1}, r_{n}t_{2}) = r_{n}U$, where $y \in A(2P, F/2)$ is from Lemma~\ref{muConverge}. Now observe that for each $n \geq N$, we have $y_{n} \not \in f(\partial U_{n})$. It then follows by Theorem~\ref{DegreeProperties}(iv) and \eqref{fnDefn} that for each $n \geq N$,
\begin{equation} \label{dnDefn}
d_{n} = \mu(y,f_{n},U) = \mu(M(r_{n},f)y,f,U_{n}) = \mu(y_{n},f,U_{n}).
\end{equation}

\begin{lem} \label{dInfty}
	Let $d_{n}$ be as in \eqref{dnDefn}. Then $d_{n} \to \infty$ as $n \to \infty$.
\end{lem}
	
\begin{proof}
	Fix some $n \geq N$ and consider $d_{n} = \mu(y_{n},f,U_{n})$ and $d_{n+1} = \mu(y_{n+1},f,U_{n+1})$. 
	
	First note that from \eqref{fnDefn} and \eqref{lowerBoundt1} we have
	\begin{equation*}
	M(t_{1}, f_{n+1}) = \frac{M(r_{n+1}t_{1},f)}{M(r_{n+1},f)} \leq 2P.
	\end{equation*}	
	 Now since $|y_{n+1}|> 2PM(r_{n+1},f) \geq M(r_{n+1}t_{1},f)$, it follows that 
	\begin{equation} \label{muStatement1}
	\mu(y_{n+1},f,A(r_{n}t_{1}, r_{n+1}t_{1})) =0.
	\end{equation}

	Next, as $|y_{n}|, |y_{n+1}| \in \left(2PM(r_{n},f), (F/2)M(r_{n+1},f)\right)$, then Theorem~\ref{DegreeProperties}(i) gives,
	\begin{equation} \label{muStatement2}
	\mu(y_{n},f,A(r_{n}t_{1},r_{n+1}t_{2})) = \mu(y_{n+1},f,A(r_{n}t_{1},r_{n+1}t_{2})).
	\end{equation}
	
	Finally, as $\min\{|f_{n}(x)| : x \in S(t_{2}) \} > F/2$, then
	\begin{equation*}
	\min\{|f(x)| : x \in S(r_{n}t_{2}) \} > \frac{F}{2}M(r_{n},f) > |y_{n}| >0.
	\end{equation*}
	
	This means by Theorem~\ref{DegreeProperties}(i),
	\begin{equation} \label{muStatement3}
	\mu(0,f,A(r_{n}t_{2}, r_{n+1}t_{2})) = \mu(y_{n},f,A(r_{n}t_{2}, r_{n+1}t_{2})).
	\end{equation}
	
	Therefore, using \eqref{muStatement1}, \eqref{muStatement2}, \eqref{muStatement3} and Theorem~\ref{DegreeProperties}(ii), 
	\begin{align} \label{dInequality}
	d_{n+1} &= d_{n+1} + \mu(y_{n+1},f,A(r_{n}t_{1}, r_{n+1}t_{1})) \nonumber \\
	&= \mu(y_{n+1},f,A(r_{n}t_{1},r_{n+1}t_{2})) \nonumber \\
	&= \mu(y_{n},f,A(r_{n}t_{1},r_{n+1}t_{2})) \nonumber \\
	&= \mu(y_{n},f,A(r_{n}t_{2},r_{n+1}t_{2})) + d_{n} \nonumber \\
	&= \mu(0,f,A(r_{n}t_{2}, r_{n+1}t_{2})) +d_{n}.
	\end{align}
	
	Now for all $n \geq N$, by applying \eqref{dInequality} finitely many times and using Theorem~\ref{DegreeProperties}(ii) again we get that, 
	\begin{equation*}
	d_{n} = \sum_{i=N}^{n-1} \mu(0,f,A(r_{i}t_{2}, r_{i+1}t_{2})) + d_{N} 
	= \mu(0,f,A(r_{N}t_{2}, r_{n}t_{2})) + d_{N}.
	\end{equation*}
	
	It remains to note that as $f$ has infinitely many zeros, then $\mu(0,f,A(r_{N}t_{2}, r_{n}t_{2})) \to \infty$ as $n \to \infty$, completing the proof.
\end{proof}
	
	 A contradiction now follows from Lemma~\ref{muConverge}, Lemma~\ref{dInfty} and \eqref{dFinite}, completing the proof of Theorem~\ref{GrowthLimitTheorem}.

\section{Proof of Theorem~\ref{MainTheorem}: Finitely many poles}

With the growth result of Theorem~\ref{GrowthLimitTheorem} established, we are now in a position to prove Theorem~\ref{MainTheorem} in the case where the quasimeromorphic mapping of transcendental type has at least one pole, but finitely many poles; this will closely follow the proof from Sections~3.2~-~3.4 in \cite{Nicks1} which covered the case for quasiregular mappings. Within the proof of Nicks, the covering and waiting sets can be found sufficiently close to the essential singularity. 

For $f:\mathbb{R}^d \to \hat{\mathbb{R}}^d$ a quasimeromorphic mapping of transcendental type with finitely many poles, there exists some $R>0$ such that all the poles of $f$ are contained in $B(0,R)$. This means that $f$ restricted to $A(R, \infty)$ is a quasiregular mapping with an essential singularity at infinity. It therefore suffices to verify that the results stated by Nicks in \cite{Nicks1} for quasiregular mappings of transcendental type on $\mathbb{R}^d$ remain valid for mappings defined on a neighbourhood of the essential singularity.

\subsection{Functions with the pits effect}

The following definition of the pits effect we shall use is adapted from \cite{BN1}. 
\begin{defn}
	Let $R>0$ and let $f : A(R, \infty) \to \mathbb{R}^d$ be a $K$-quasiregular mapping with an essential singularity at infinity. Then $f$ is said to have the pits effect if there exists some $N \in \mathbb{N}$ such that, for all $s>1$ and all $\epsilon>0$, there exists $T_{0} \geq R$ such that 
	\begin{equation*}
	\{ x \in \overline{A(T,sT)} : |f(x)| \leq 1 \}
	\end{equation*} 
can be covered by $N$ balls of radius $\epsilon T$ whenever $T > T_{0}$.
\end{defn}

As a direct consequence of \cite[Theorem~8.1]{BN1}, using Corollary~\ref{GrowthLimitTheoremCorollary} rather than \cite[Lemma~3.4]{Bergweiler2} in the proof, we get the following analogous result.
\begin{lem}	\label{pitsDefn2}
Let $R>0$ and let $f : A(R, \infty) \to \mathbb{R}^d$ be a $K$-quasiregular mapping with an essential singularity at infinity that has the pits effect. Then there exists some $N \in \mathbb{N}$ such that, for all $s>1$, all $\alpha>1$ and all $\epsilon>0$, there exists $T_{0} \geq R$ such that 
	\begin{equation*}
	\{ x \in \overline{A(T,sT)} : |f(x)| \leq T^{\alpha} \}
	\end{equation*} 
can be covered by $N$ balls of radius $\epsilon T$ whenever $T > T_{0}$.
\end{lem}

Throughout the remainder of Section~4.1, we shall assume that $f$ is as in the statement of Theorem~\ref{MainTheorem} and that the restriction $f:A(R,\infty) \to \mathbb{R}^d$ is a $K$-quasiregular mapping that has the pits effect. Using Lemma~\ref{R'lemma}, we can further assume that $R>0$ is sufficiently large enough that $M(r,f)$ is a strictly increasing function for $r\geq R$.

First we require some self-covering sets to achieve the `hold-up' criteria from Lemma~\ref{lem:HoldUp}. The following lemma is essentially that of \cite[Lemma~3.3]{Nicks1}, with the proof following similarly.

\begin{lem} \label{pitsCoverN}
There exists $\delta \in \left(0,1/2 \right]$ and a sequence of points $x_{t} \to \infty$ such that the moduli $T_{t} = |x_{t}|$ are strictly increasing and the balls $B_{t} := B(x_{t}, \delta T_{t})$ are such that
	\begin{equation} \label{pitsHoldup}
	B_{t} \subset B(0, 2T_{t}) \subset f(B_{t})
	\end{equation}
for all $t \in \mathbb{N}$.
\end{lem}

From Corollary~\ref{GrowthLimitTheoremCorollary}, for all large $r$ we have $M(r,f) > 2r$. Thus we shall now assume that the $T_{t}$ as defined in Lemma~\ref{pitsCoverN} are large enough such that the sequence $(r_{t})$, defined by $M(r_{t},f) = T_{t}$ with $r_{t} > \max\{R,M(R,f)\}$, satisfies $M(r_{t},f) > 2r_{t}$ for all $t \in \mathbb{N}$. Consequently, note that $(r_{t})$ is a strictly increasing sequence with $r_{t} \to \infty$ as $t \to \infty$. We now have the following result, which is based on \cite[Lemma~3.4]{Nicks1} and whose proof also follows similarly.

\begin{lem} \label{pitsMovement}
For each $t \in \mathbb{N}$ and $\lambda \geq 2T_{t}$, 
	\begin{equation*}
	A(r_{t},2\lambda) \subset f(A(r_{t},\lambda)).
	\end{equation*}
\end{lem}

Using Lemma~\ref{pitsCoverN} and Lemma~\ref{pitsMovement}, we can appeal to Lemma~\ref{lem:HoldUp}, with $p =1$, to complete the proof of Theorem~\ref{MainTheorem} for mappings with finitely many poles that have the pits effect. With this in mind, we shall omit the superscripts and choose the sets $X_{m}$ for each $m$. 

Set $m_{1} = 1$ and inductively define $m_{t+1} = m_{t} + K_{t}$, where $K_{t} >1$ is the smallest integer such that $(3/2)T_{t+1} \leq 2^{K_{t}}T_{t}$. Now for each $m \in \mathbb{N}$, set 
\[ X_{m} = \left\{
\begin{array}{cl}
B_{t} & \text{ if } m = m_{t} \text{ for some } t \in \mathbb{N}; \\
A(r_{t}, 2^{m-m_{t}}T_{t}) & \text{ if } m \in (m_{t}, m_{t+1}).
\end{array}
\right.
\]

Firstly note that as $T_{t} \to \infty$ and $r_{t} \to \infty$ as $t \to \infty$, then \eqref{SmallestXToInfty} is satisfied. In addition, (X2) is satisfied due to \eqref{pitsHoldup} from Lemma~\ref{pitsCoverN}. Next as $T_{t}$ are large, then from Theorem~\ref{JuliaSetProperties}(i) we can assume that $B(0, 2T_{t}) \cap J(f) \neq \varnothing$. From this, \eqref{pitsHoldup} and Theorem~\ref{JuliaSetProperties}(iii) then imply that $B_{t} \cap J(f) \neq \varnothing$, so (X3) is satisfied. To show (X1) holds, we shall consider three cases:

\begin{enumerate}
\item[(1)] When $m = m_{t}$ for some $t \in \mathbb{N}$, then from \eqref{pitsHoldup},
	\begin{equation*}
	f(X_{m_{t}}) = f(B_{t}) \supset B(0, 2T_{t}) \supset A(r_{t}, 2T_{t}) = X_{m_{t} +1}.
	\end{equation*}
	
\item[(2)] When $m \in (m_{t}, m_{t+1} -1)$ for some $t \in \mathbb{N}$, then by Lemma~\ref{pitsMovement},
	\begin{equation*}
	f(X_{m}) = f(A(r_{t}, 2^{m-m_{t}}T_{t})) \supset A(r_{t}, 2^{m+1-m_{t}}T_{t}) = X_{m+1}.
	\end{equation*}
	
\item[(3)] When $m = m_{t+1} -1$ for some $t \in \mathbb{N}$, then by Lemma~\ref{pitsMovement},
	\begin{align*}
	f(X_{m}) &= f(A(r_{t}, 2^{m_{t+1} - 1 -m_{t}}T_{t})) \supset A(r_{t}, 2^{m_{t+1}-m_{t}}T_{t}) \\
	&= A(r_{t}, 2^{K_{t}}T_{t}) \supset A\left(r_{t}, \frac{3T_{t+1}}{2}\right).
	\end{align*}

Now since $T_{t+1} \geq T_{t} > 2r_{t}$ for all $t$, then  
	\begin{equation*}
	f(X_{m}) \supset A\left(r_{t}, \frac{3T_{t+1}}{2}\right) \supset A\left(\frac{T_{t+1}}{2}, \frac{3T_{t+1}}{2}\right) \supset B_{t+1} = X_{m+1}.
	\end{equation*}
\end{enumerate}

Finally, as all the hypotheses are satisfied, then an application of Lemma~\ref{lem:HoldUp} completes the proof of Theorem~\ref{MainTheorem} for mappings with finitely many poles that have the pits effect.

\subsection{Functions without the pits effect}
In this subsection, the main objective is to prove Theorem~\ref{MainTheorem} in the case where $f:\mathbb{R}^d \to \hat{\mathbb{R}}^d$ is a quasimeromorphic function of transcendental type with finitely many poles, whose restriction to a domain near the essential singularity is a quasiregular mapping that does not have the pits effect. This will be done by adapting the methods found in \cite[Section~3.4]{Nicks1}.

For $r>4R>0$, we shall first define domains $Q_{\ell}(r) \subset A(R, \infty)$. In the following, we use the notation $rX := \{rx: x \in X\}$.

Let $q \in \mathbb{N}$ and fix $2q$ distinct unit vectors $\hat{u_{1}}, \hat{u_{2}}, \dots, \hat{u_{2q}}$, so each $\hat{u_{\ell}}$ is such that $\hat{u_{\ell}} \in \mathbb{R}^d$ and $|\hat{u_{\ell}}| = 1$. Fix $\theta >0$ small enough so for all $\ell = 1,2,\dots,2q$, the truncated cones
\begin{equation*}
C_{\ell} = \left\{ x \in A\left(\frac{1}{4}, 2q+1\right) : \frac{\hat{u_{\ell}} \cdot x}{|x|} > \cos(\theta)\right\}
\end{equation*}
are such that $\overline{C_{\ell}} \cap \overline{C_{j}} = \varnothing$ for all pairs $\ell \neq j$, where $\hat{u_{\ell}} \cdot x$ is the scalar product.

Now for $r>4R$ and $\ell \in \{1,2,\dots,2q\}$, define
\begin{equation} \label{QirDefn}
Q_{\ell}(r) = A\left(\ell r, \left(\ell + \frac{1}{2}\right)r\right) \cup rC_{\ell}.
\end{equation}

A useful observation is that for all $\ell$ and $r$, $Q_{\ell}(r) = rQ_{\ell}(1)$ and that each $Q_{\ell}(1)$ is bounded away from infinity by the chordal metric.

By using a combinatorial argument, we can get a useful extension of Lemma~\ref{qmMontel}. Here, we shall state the result for a family of $K$-quasiregular mappings, however the proof is analogous in the quasimeromorphic case.

\begin{lem} \label{omitQirNormal}
Let $\mathcal{F}$ be a family of $K$-quasiregular mappings on a domain $X \subset \mathbb{R}^d$ and let $q_{0}$ be Rickman's quasiregular constant. Let $N \in \mathbb{N}$ and, for $i = 1,2,\dots,Nq_{0}$ and $n = 1,2,\dots, N$, let $A_{i,n}$ be bounded sets such that for each $n$, $\overline{A_{i,n}} \cap \overline{A_{j,n}} = \varnothing$ for all $i \neq j$.  

Suppose that every $g \in \mathcal{F}$ omits a value from each set $\mathcal{A}_{i} = \bigcup_{n=1}^N A_{i,n}$. Then $\mathcal{F}$ is a normal family on $X$.
\end{lem}

\begin{proof}
Fix an $N \in \mathbb{N}$ and for each $n \in \{1,2,\dots,N\}$, let $\epsilon_{n}>0$ be such that, for all $i \neq j$, 
	\begin{equation*}
	\dist(A_{i,n}, A_{j,n}) := \inf\{|a_{i} - a_{j}| : a_{i} \in A_{i,n}, a_{j} \in A_{j,n}\} \geq \epsilon_{n}.
	\end{equation*} 

Set $\epsilon = \min\{\epsilon_{n}: n = 1,2,\dots,N\}$ and consider any set $D = \{d_{1}, d_{2}, \dots, d_{Nq_{0}}\}$, where $d_{i} \in \mathcal{A}_{i}$ for each $i$. Since each $d_{i}$ must belong to one of the sets $A_{i,n}$ with $n \in \{1,2,\dots,N\}$, then by the pigeonhole principle there must exist some subset $\{\alpha_{1}, \alpha_{2}, \dots, \alpha_{q_{0}}\} \subset D$ such that $|\alpha_{k} - \alpha_{l}| \geq \epsilon$ for $k \neq l$. Now by considering Lemma~\ref{qmMontel} and noting that each of the $A_{i,n}$ are bounded away from infinity in the chordal metric, we conclude that $\mathcal{F}$ is a normal family on $X$.
\end{proof}

Note that in the above theorem, the result can be sharpened by asking that every mapping in $\mathcal{F}$ omits a value in at least $N(q_{0}-1)+1$ of the $\mathcal{A}_{i}$. We shall apply this lemma later with $N=2, A_{i,1}= A(i, i+ 1/2), A_{i,2} = C_{i}$ so that $\mathcal{A}_{i} = Q_{i}(1)$.

To find sets that satisfy the `hold-up' criterion, we will first introduce some notation. Following Rickman \cite[p.~80]{Rickman3}, using the Brouwer degree in \eqref{muDefn} we define
\begin{equation*}
AV(f,D) :=\frac{1}{\omega_{d}}\int_{\mathbb{R}^d}\frac{\mu(y, f,D)}{(1+|y|^2)^d}dy = \frac{1}{\omega_{d}}\int_{D}\frac{J_{f}(x)}{(1+|f(x)|^2)^d}dx,
\end{equation*} 
which is the average of $\mu(y,f,D)$ over all $y \in \hat{\mathbb{R}}^d$. Here $\omega_{d}$ denotes the surface area of the unit $d$-sphere $S^{d}(0,1)$. It should be noted that Rickman identifies $\hat{\mathbb{R}}^d$ with $\{x \in \mathbb{R}^{d+1} : |x-(1/2)e_{d+1}| = 1/2 \}$, where $e_{k}$ denotes the $k^{\text{th}}$ unit vector, while we use $\{x \in \mathbb{R}^{d+1} : |x| = 1 \}$. This accounts for the differing factor of $2^d$ in the above definition.

By utilising the average Brouwer degree, we can give a criterion which states that if we have sufficiently many bounded domains such that the image of each one covers many of the others, then the closure of each domain must intersect the Julia set. This is an extension of \cite[Lemma~2.5]{Nicks1} to the case of quasiregular mappings defined near an essential singularity.

\begin{lem} \label{JuliaSetIntersectionWithoutPits}
 Let $f: \mathbb{R}^d \to \hat{\mathbb{R}}^d$ be a $K$-quasimeromorphic mapping of transcendental type with at least one pole. Let $p \in \mathbb{N}$ be such that $p > K_{I}(f) + q_{0}$, where $q_{0}$ is Rickman's quasiregular constant. Suppose that $W_{1}, W_{2}, \dots, W_{p} \subset \mathbb{R}^d$ are bounded domains such that $\overline{W_{i}} \cap \overline{W_{j}} = \varnothing$ for all $i \neq j$, and for each $i \in \{1,2,\dots,p\}$,
\begin{equation*}
f(W_{i}) \supset W_{j} \text{ for at least } p-q_{0} \text{ values of } j\in \{1,2,\dots,p\}.
\end{equation*}
Then $\overline{W_{i}} \cap J(f) \neq \varnothing$ for all $i \in \{1,2,\dots,p\}$.
\end{lem}

\begin{proof}
Firstly, suppose that $J(f) \cap \overline{W_{i}} = \varnothing$ for some $i \in \{1,2,\dots,p\}$. Then $W_{i} \cap \mathcal{O}^{-}_{f}(\infty) = \varnothing$, so $f^n$ is $K$-quasiregular on $W_{i}$ for all $n \in \mathbb{N}$. Now note that for any $n \in \mathbb{N}$ then, counting multiplicity, $f^n(W_{i})$ covers at least $(p-q_{0})^n$ of the domains $W_{j}$, $j \in \{1,2,\dots,p\}$. By setting $\nu = (p-q_{0})^n$, there exist pairwise disjoint subsets $V_{1}, V_{2}, \dots, V_{\nu}$ of $W_{i}$ such that if $m \in \{1,2,\dots, \nu \}$, then $f^n(V_{m})=W_{j}$ for some $j \in \{1,2,\dots,p\}$. Hence for each $n \in \mathbb{N}$, there exists some $j \in \{1,2,\dots,p\}$ such that
\begin{equation*}
\mu(y,f^n, W_{i}) \geq \frac{\nu}{p} \text{ for all } y \in W_{j}.
\end{equation*}

This implies that there exists some constant $C_{1}>0$ such that for all $n \in \mathbb{N}$,
\begin{equation}\label{ALowerBound}
AV(f^n,W_{i}) \geq \frac{C_{1}\nu}{p}.
\end{equation}

Now as $J(f) \cap \overline{W_{i}} = \varnothing$, then for each $x \in \overline{W_{i}}$ there exists some $\delta_{x}>0$ such that $B(x, 2\delta_{x}) \cap J(f) = \varnothing$ and $\hat{\mathbb{R}}^d \setminus \mathcal{O}_{f}^{+}(B(x, 2\delta_{x}))$ is infinite. This means that there exists a non-exceptional point $y \in \hat{\mathbb{R}}^d \setminus \left(\mathcal{O}_{f}^{+}(B(x, 2\delta_{x})) \cup E(f)\right)$. Since $\hat{\mathbb{R}}^d \setminus \mathcal{O}_{f}^{+}(B(x, 2\delta_{x}))$ is closed, then $ \overline{\mathcal{O}_{f}^{-}(y)} \subset \hat{\mathbb{R}}^d \setminus \mathcal{O}_{f}^{+}(B(x, 2\delta_{x}))$. As $y \not \in E(f)$, it follows by Theorem~\ref{EAndCapacity} and the definition of the forward orbit that $\capacity\left(\hat{\mathbb{R}}^d \setminus \mathcal{O}_{f}^{+}(B(x, 2\delta_{x}))\right)>0$.

Using \cite[Theorem 3.2]{Bergweiler6} and \eqref{compositionDilatation}, for each $x \in \overline{W_{i}}$ there exists some constant $C_{x} >0$, dependent on $x$, such that for all $n \in \mathbb{N}$, 
\begin{equation*}
AV(f^n, B(x, \delta_{x})) \leq C_{x}K_{I}(f^n) \leq C_{x}K_{I}(f)^n.
\end{equation*}
As $\overline{W_{i}}$ is compact and the union of $B(x, \delta_{x})$ forms an open cover, then there exists a finite subcover of $\overline{W_{i}}$. Thus we get that there exists some constant $C_{2}>0$ such that
\begin{equation}\label{AUpperBound}
AV(f^n, W_{i}) \leq C_{2}K_{I}(f)^n.
\end{equation}

However as $p > K_{I}(f) + q_{0}$, then we get a contradiction from \eqref{ALowerBound} and \eqref{AUpperBound} when $n \in \mathbb{N}$ is large. The conclusion now follows.
\end{proof}

Now by appealing to Lemma~\ref{R'lemma} and Corollary~\ref{GrowthLimitTheoremCorollary}, throughout the remainder of Section~4.2 we assume without loss of generality that $R>0$ is sufficiently large such that the restriction $f:A(R,\infty) \to \mathbb{R}^d$ is a $K$-quasiregular mapping with an essential singularity at infinity that does not have the pits effect, and $M(r,f)$ is a strictly increasing function with $M(r,f)>r$ for all $r\geq R$. 

The covering result will be based on that given in \cite{Nicks1}; the proof follows analogously using the new growth condition of Theorem~\ref{GrowthLimitTheorem}.

\begin{lem} \label{coveringNoPits}
Let $q_{0}$ be Rickman's quasiregular constant and let $W_{1}, W_{2}, \dots, W_{q_{0}} \subset \mathbb{R}^d$ be bounded sets such that $\overline{W_{i}} \cap \overline{W_{j}} = \varnothing$ for all pairs $i \neq j$. Then for all sufficiently large $r$ and each $\ell = 1,2,\dots, 2q_{0}$, the following hold.
	\begin{enumerate}
	\item [\emph{(C1)}] There exists some $j \in \{1,2,\dots,2q_{0}\}$ such that $f(Q_{\ell}(r)) \supset Q_{j}(M(r,f))$; 
	\item [\emph{(C2)}] There exists some $k \in \{1,2,\dots,q_{0}\}$ such that $f(Q_{\ell}(r)) \supset M(r,f)W_{k}$. 
	\end{enumerate}
\end{lem}

The `hold-up' lemma we will use is also closely based on \cite[Section~3]{BN1} (see also \cite[Lemma~3.7]{Nicks1}). We omit the proof here, noting that the adapted proof uses Lemma~\ref{JuliaSetIntersectionWithoutPits} and Theorem~\ref{GrowthLimitTheorem}.

\begin{lem} \label{holdUpNoPits}
Let $q_{0}$ be Rickman's quasiregular constant. Then there exist bounded domains $W_{1}, W_{2}, \dots, W_{q_{0}} \subset \mathbb{R}^d$ with $\overline{W_{i}} \subset \left\{x \in \mathbb{R}^d : |x| \geq 1/2 \right\}$  satisfying $\overline{W_{i}} \cap \overline{W_{j}} = \varnothing$ for all pairs $i \neq j$, and a real sequence $T_{t} \to \infty$ with $T_{1}>4R$ such that for every $t \in \mathbb{N}$ and $\ell \in \{1,2,\dots, q_{0}\}$ the following hold:
	\begin{enumerate}
	\item [\emph{(C3)}] There exists some $j \in \{1,2,\dots,q_{0}\}$ such that $f(T_{t}W_{\ell}) \supset T_{t}W_{j}$; 
	\item [\emph{(C4)}] For each $\alpha \in [4R, M(T_{t},f)]$, there exists some $k \in \{1,2,\dots,2q_{0}\}$ such that $f(T_{t}W_{\ell}) \supset Q_{k}(\alpha)$;
	\item [\emph{(C5)}] $T_{t}\overline{W_{\ell}} \cap J(f) \neq \varnothing$.
	\end{enumerate}
\end{lem}

Now using Lemma~\ref{coveringNoPits} and Lemma~\ref{holdUpNoPits}, we shall once again appeal to Lemma~\ref{lem:HoldUp} to complete the proof of Theorem~\ref{MainTheorem} for mappings with finitely many poles that do not have the pits effect. This closely follows the construction technique in \cite[Section~3.4]{Nicks1}. 

Recall that $R>0$ is sufficiently large such that $f: A(R, \infty) \to \mathbb{R}^d$ is a $K$-quasiregular mapping with an essential singularity at infinity and $M(r,f)$ is a strictly increasing function with $M(r,f)>r$ for all $r \geq R$. Now for $p \in \mathbb{N} \cup \{0\}$, we define the iterated maximum modulus $M^p(r,f)$ as follows. Set $M^{0}(r,f) = r$ and $M^{1}(r,f) = M(r,f)$. Then for $p \geq 2$, iteratively define
\begin{equation*}
M^{p}(r,f) = M(M^{p-1}(r,f),f).
\end{equation*}
We note that as $M(r,f)>r$ is strictly increasing on $r\geq R$, then $M^p(r,f)>r$ for all $p \in \mathbb{N} \cup \{0\}$, so these are well-defined for $f$.

Now towards the proof, first take a real sequence $T_{t} \to \infty$ and bounded domains $W_{1}, W_{2}, \dots, W_{q_{0}}$ as in Lemma~\ref{holdUpNoPits}. We may assume that $T_{t+1} > M(T_{t},f)$ and, by Corollary~\ref{GrowthLimitTheoremCorollary} for example, we may assume that $T_{1}>4R$ is large enough such that $M^{p}(T_{1},f) \to \infty$ as $p \to \infty$.
Then for each $t \in \mathbb{N}$, there exists a smallest integer $p_{t} \geq 2$ such that $M^{p_{t}}(T_{t},f) \geq T_{t+1}$. Now by our choice of $p_{t}$, we have that $M^{p_{t}-1}(r,f)$ is continuous in $r$ and
\begin{equation*}
M^{p_{t}-1}(T_{t},f) \leq T_{t+1} \leq M^{p_{t}}(T_{t},f) = M^{p_{t}-1}(M(T_{t},f),f).
\end{equation*}
Hence by the intermediate value theorem, for each $t \in \mathbb{N}$ there exists some $\Upsilon_{t} \in [T_{t}, M(T_{t},f)]$ such that $M^{p_{t}-1}(\Upsilon_{t},f) = T_{t+1}$

We now choose the sets $X^{(i)}_{m}$ for each $m \in \mathbb{N}$ and $i = 1,2,\dots,2q_{0}$ to satisfy Lemma~\ref{lem:HoldUp} with $p = 2q_{0}$. Set $m_{1} = 1$ and inductively define $m_{t+1} = m_{t} + p_{t}$, for $t \geq 1$. Now for each $m \in \mathbb{N}$ and for each $i = 1,2,\dots,2q_{0}$, set

\[ X^{(i)}_{m} = \left\{
\begin{array}{cl}
T_{t}W_{i} & \text{ if } m = m_{t} \text{ for some } t \in \mathbb{N}, i \leq q_{0}; \\
T_{t}W_{1} & \text{ if } m = m_{t} \text{ for some } t \in \mathbb{N}, i > q_{0}; \\
Q_{i}(M^{m - m_{t}-1}(\Upsilon_{t},f)) & \text{ if } m \in (m_{t}, m_{t+1}).
\end{array}
\right.
\]

Firstly note that as the $W_{i}$ and $T_{t}$ were chosen to be those from Lemma~\ref{holdUpNoPits}, then $T_{t}>4R$ for each $t \in \mathbb{N}$ and $W_{i} \subset \{x \in \mathbb{R}^d : |x| \geq 1/2 \}$ for each $i \in \{1,2,\dots,2q_{0}\}$. This means that
\begin{equation*}
\inf\{|x|: x \in X_{m_{t}}\} = \inf \left\{ |x| : x \in \bigcup^{2q_{0}}_{i=1}T_{t}W_{i} \right\} \geq \frac{T_{t}}{2}.
\end{equation*}

Also by the definition of $Q_{i}(r)$, then for $m \in (m_{t}, m_{t+1})$ we have
\begin{equation*}
\inf\{|x|: x \in X_{m}\} = \frac{M^{m - m_{t}-1}(\Upsilon_{t},f)}{4} \geq \frac{\Upsilon_{t}}{4} \geq \frac{T_{t}}{4}.
\end{equation*}

Since $T_{t} \to \infty$ as $t \to \infty$, then \eqref{SmallestXToInfty} is satisfied. Further, observe that (X2) and (X3) are satisfied due to (C3) and (C5) from Lemma~\ref{holdUpNoPits} respectively. Finally (X1) follows from (C1) and (C2) from Lemma~\ref{coveringNoPits}, and (C4) from Lemma~\ref{holdUpNoPits}; see \cite{Nicks1} for details.

As all the hypotheses are satisfied, then an application of Lemma~\ref{lem:HoldUp} completes the proof of Theorem~\ref{MainTheorem} for mappings with finitely many poles that do not have the pits effect.

\subsection{A covering result for functions without the pits effect}

Let $f: \mathbb{R}^d \to \hat{\mathbb{R}}^d$ be a quasimeromorphic function without the pits effect as in Section~4.2. By continuing to adopt the notation as in Section~4.2, we shall give a useful covering result regarding the sets $T_{t}W_{j}$ for use in Section~6.

\begin{lem}\label{BU(f)(ii)}
For $t \in \mathbb{N}$ and $1 \leq j \leq q_{0}$, let $T_{t}$ and $W_{j}$ be those from the proof of Theorem~\ref{MainTheorem}. By moving to a subsequence of $(T_{t})$, there exist constants $i_{0},j_{0} \in \{1,2,\dots,q_{0}\}$ and $d \in \mathbb{N}$ such that for each $t \in \mathbb{N}$, there is some $c_{t} \in \mathbb{N}$, some subset $Y_{t} \subset T_{2}W_{j_{0}}$ and some subset $Z_{t} \subset T_{t+2}W_{i_{0}}$ where
\begin{equation*}
f^{c_{t}}(Y_{t}) \supset T_{t+2}W_{i_{0}} \text{ and } f^{d}(Z_{t}) \supset T_{2}W_{j_{0}}.
\end{equation*}
\end{lem}

\begin{proof}
Let $j \in \{1,2,\dots,q_{0}\}$. By the construction after Lemma~\ref{holdUpNoPits}, it follows that for all $t \in \mathbb{N}$ there exists some $i_{j,t} \in \{1,2,\dots,q_{0}\}$, some $c_{j,t} \in \mathbb{N}$ and some subset $Y_{j,t} \subset T_{2}W_{j}$ such that
\begin{equation*}
f^{c_{j,t}}(Y_{j,t}) \supset T_{t+2}W_{i_{j,t}}.
\end{equation*}

Since $i_{j,t}$ can only take values from a finite set, then by taking a suitable subsequence of $T_{t}$ and relabelling we can assume that $i_{j}=i_{j,t}$ is independent of $t \in \mathbb{N}$, so
\begin{equation}\label{(ii)Main1}
f^{c_{j,t}}(Y_{j,t}) \supset T_{t+2}W_{i_{j}}.
\end{equation}

Next, observe that as $T_{2}>M(T_{1},f)$, then there exists some $\alpha>T_{1}>4R$ such that $M(\alpha,f)=T_{2}$. Then by (C4) from  Lemma~\ref{holdUpNoPits}, for all $t \in \mathbb{N}$ there exists some $N_{j,t} \in \{1,2,\dots, 2q_{0}\}$ such that $f(T_{t+2}W_{i_{j}}) \supset Q_{N_{j,t}}(\alpha)$. As $N_{j,t}$ can only take values from a finite set, then by taking another suitable subsequence of $T_{t}$ and relabelling, we can assume that $N_{j}=N_{j,t}$ is independent of $t$. This means that for all $t \in \mathbb{N}$,
\begin{equation}\label{QCover1}
f(T_{t+2}W_{i_{j}}) \supset Q_{N_{j}}(\alpha).
\end{equation}

Applying (C2) from Lemma~\ref{coveringNoPits}, we get that there exists some $\ell\in \{1,2,\dots,q_{0}\}$ such that
\begin{equation}\label{QCover3}
f(Q_{N_{j}}(\alpha)) \supset M(\alpha,f)W_{\ell} = T_{2}W_{\ell}.
\end{equation}

Combining \eqref{QCover1} and \eqref{QCover3}, it follows that for each $t \in \mathbb{N}$ there exists some subset $P_{j,t} \subset T_{t+2}W_{i_{j}}$ and some $\ell \in \{1,2,\dots,q_{0}\}$ such that 
\begin{equation}\label{(ii)Main2}
f^{2}(P_{j,t}) \supset T_{2}W_{\ell}.
\end{equation}

By repeatedly applying the whole argument above, we can build a sequence of subscripts $(\ell_{n})$ as follows. Set $\ell_{1}=1$. Then for each $n \geq 1$, let $\ell_{n+1}$ be the value of $\ell$ from \eqref{(ii)Main2} after applying the argument once to $T_{2}W_{\ell_{n}}$. 

As $\ell_{n} \in \{1,2,\dots,q_{0}\}$ for all $n \in \mathbb{N}$, then there will exist some smallest values $n_{1}, n_{2} \in \mathbb{N}$, with $n_{1}<n_{2}$, such that $\ell_{n_{1}}=\ell_{n_{2}}$. Using this, we set $j_{0}=n_{1}$ and $i_{0}=i_{n_{1}}$. Then for each $t \in \mathbb{N}$, set $Y_{t} = Y_{n_{1},t}$, set $c_{t}=c_{n_{1},t}$ and $d = 2(n_{2}-n_{1}) + \sum_{j=n_{1}+1}^{n_{2}-1}c_{j,1}$. It follows that there is some subset $Z_{t} \subset T_{t+2}W_{i_{n_{1}}}$ such that
\begin{equation*}
f^{c_{t}}(Y_{t}) \supset T_{t+2}W_{i_{n_{1}}} \text{ and } f^{d}(Z_{t}) \supset T_{2}W_{n_{1}}
\end{equation*}
as required.
\end{proof}

\section{Proof of Theorem~\ref{MainTheorem}: Infinitely many poles}

In the case where $f$ has an infinite number of poles, it makes sense to utilise the neighbourhoods of the poles as a means of naturally approaching infinity. To this end, we seek a point that is able to `pole-hop' between each neighbourhood and is able to return to the same neighbourhood after a finite number of steps via bounded sets. This idea is similar to that used by Rippon and Stallard in \cite[Section~4]{RS5}, however the execution is quite different as it does not rely on the Ahlfors five island theorem. 

The `pole-hop' method creates a different situation to that found in the case of finitely many poles, where instead we relied on finding a point that could move forward at any time from any set. To achieve this modified `hold-up' condition, we need to establish a different version of Lemma~\ref{lem:HoldUp}. 

For $m \in \mathbb{N}$ and some fixed $p \in \mathbb{N}$, we shall denote the residue $m \pmod p \in \{0,1,2,\dots,p-1\}$ as $[m]_{p}$. 

\begin{lem} \label{lem:HoldUp2}
	Let $f:\mathbb{R}^d \to \hat{\mathbb{R}}^d$ be a quasimeromorphic function of transcendental type with at least one pole. Let $p \in \mathbb{N}$ and for $m \in \mathbb{N}$ and $i \in \{1,2,\dots, p\}$, let $X_{m}^{(i)} \subset \mathbb{R}^d$ be non-empty bounded sets, with $X_{m} = \bigcup_{i=1}^{p} X_{m}^{(i)}$, such that
	\begin{equation} \label{SmallestXToInfty2}
	\inf\{|x| : x \in X_{m} \} \to \infty \text{ as } m \to \infty.
	\end{equation}
	
	Suppose further that
	\begin{enumerate}
	\item [\emph{(X4)}] for all $m \in \mathbb{N}$, $f\left(X_{m}^{(1)} \right) \supset X_{m+1}^{(1)}$,
	\end{enumerate}
	and there exists a strictly increasing sequence of integers $(m_{t})$ such that 
	\begin{enumerate}
	\item [\emph{(X5)}] for all $t \in \mathbb{N}$ and $i \in \{1,2,\dots, p\}$, $f\left(X_{m_{t}}^{(i)} \right) \supset X_{m_{t}}^{([i]_{p} + 1)}$, and
	\item [\emph{(X6)}] for all $t \in \mathbb{N}$ and $i \in \{1,2,\dots, p\}$, $\overline{X_{m_{t}}^{(i)}} \cap J(f) \neq \varnothing$.
	\end{enumerate}
	
	Then given any positive sequence $a_n \to \infty$, there exists $\zeta \in J(f)$ and $N_{1} \in \mathbb{N}$ such that $|f^n(\zeta)| \to \infty$ as $n \to \infty$, while also $|f^n(\zeta)| \leq a_n$ whenever $n \geq N_{1}$.
\end{lem}

\begin{proof}
	Define an increasing real sequence $(\gamma_{m})$ by
	\begin{equation} \label{gammaDefn2}
	\gamma_{m} = \sup\left\{|x| : x \in \bigcup_{j=1}^{m} X_{j} \right\}
	\end{equation} 
	
	Since $a_{n} \to \infty$, then we can define a strictly increasing sequence of integers $N_{t}$ such that $\gamma_{m_{t}} \leq a_{n}$ for all $n \geq N_{t}$.

We shall now inductively define sets $F_{n}$, with $n \geq N_{1}$. Set $F_{N_{1}} = X^{(1)}_{m_{1}}$ and for each integer $n \geq N_{1}$, define 

\[
 F_{n+1} = 
  \begin{cases} 
   X_{m}^{([i]_{p} +1)} 		& \text{if } F_{n} = X_{m}^{(i)}, i \neq 1 \\
   X_{m +1}^{(1)}      			& \text{if } F_{n} = X_{m}^{(1)}, m \neq m_{t} \\
   X_{m_{t} +1}^{(1)}			& \text{if } F_{n} = X_{m_{t}}^{(1)}, n \geq N_{t+1} \\
   X_{m_{t}}^{([1]_{p} +1)}     & \text{if } F_{n} = X_{m_{t}}^{(1)}, n < N_{t+1}.
  \end{cases}
\]

Firstly, observe that if $F_{n} = X_{m}^{(i)}$ with $i \neq 1$, then $m=m_{t}$ for some $t \in \mathbb{N}$. For supposing otherwise, then by construction there exists some natural number $1 \leq k <p$ such that $F_{n-k} = X_{m}^{(1)}$. If $m \neq m_{t}$ for any $t \in \mathbb{N}$, it follows that $F_{n-k+1} = X_{m+1}^{(1)}$. However, this is a contradiction since $n-k+1\leq n$ and $F_{n} = X_{m}^{(i)}$, but $m+1>m$. 

Now it follows from the construction, (X4) and (X5) that for each $n \geq N_{1}$, then $f(F_{n}) \supset F_{n+1}$. From this, together with (X6), then by Lemma~\ref{NestedSeqn} there exists a point $\zeta_{N_{1}} \in J(f) \setminus \{\infty\}$ such that $f^{n-N_{1}}(\zeta_{N_{1}}) \in \overline{F_{n}}$ for all $n \geq N_{1}$. 

Without loss of generality, we may assume that $\zeta_{N_{1}} \not \in E(f)$. By applying Theorem~\ref{PicardTheorem} finitely many times and noting Theorem~\ref{JuliaSetProperties}(iii), it follows that there exists $\zeta \in J(f)$ such that $f^{N_{1}}(\zeta) = \zeta_{N_{1}}$. Therefore we have that $f^{n}(\zeta) \in \overline{F_{n}}$ for all $n \geq N_{1}$. Further, by \eqref{SmallestXToInfty2} we have that $|f^n(\zeta)| \to \infty$ as $n \to \infty$.

To complete the proof, it remains to show that for all $n \geq N_{1}$, then $|f^{n}(\zeta)| \leq a_{n}$. Indeed, let $n \geq N_{1}$ be such that $F_{n} = X_{m_{1}}^{(i)}$ for some $i \in \{1,2,\dots,p\}$. Then $F_{n} \subset X_{m_{1}}$ and so by \eqref{gammaDefn2} and the definition of $N_{1}$,
\begin{equation}\label{supPart1}
\sup\{|x|: x \in F_{n}\} \leq \gamma_{m_{1}} \leq a_{n}.
\end{equation}

We next aim to prove the following claim. Suppose that $n >N_{1}$ and $t \in \mathbb{N}$ are such that $m_{1} \leq m_{t} < m \leq m_{t+1}$ and $F_{n} = X_{m}^{(i)}$ for some $i \in \{1,2,\dots,p\}$. Then $n \geq N_{t+1}$.

Indeed, if $i \neq 1$, then by a previous observation we must have $m=m_{t+1}$. This means there exists some natural number $k<p$ such that $F_{n-k}=X_{m}^{(1)}$ and $n-k> N_{1}$. Hence for any $i \in \{1,2,\dots,p\}$, there exists some $N_{1}< n_{1}\leq n$ such that $F_{n_{1}} = X_{m}^{(1)}$.

It follows by construction that either $F_{n_{1}-1} = X_{m-1}^{(1)}$ or $F_{n_{1}-1} = X_{m}^{(p)}$, where the latter case occurs only if $m=m_{t+1}$. As $m>m_{1}$, then by applying the above argument finitely many times, there must exist some integer $r \geq 0$ such that $F_{n_{1}-rp} = X_{m}^{(1)}$ and $F_{n_{1}-rp-1} = X_{m-1}^{(1)}$. It should be noted here that $n_{1}-rp >N_{1}$ as $m > m_{1}$. Hence there exists some $N_{1}<n_{2}\leq n_{1}$ such that $F_{n_{2}} = X_{m}^{(1)}$ and $F_{n_{2}-1} = X_{m-1}^{(1)}$.

As $F_{n_{2}} = X_{m}^{(1)}$ and $F_{n_{2}-1} = X_{m-1}^{(1)}$, then one of two cases may arise. If $m-1=m_{t}$, then this can only happen if $n_{2} \geq N_{t+1}$ by construction. Hence in this case, $n \geq N_{t+1}$.

If $m-1 \neq m_{t}$, then by construction we can find some $N_{1} \leq n_{3} < n_{2}$ such that $F_{n_{3}} = X_{m_{t}}^{(1)}$ and $F_{n_{3}+1} = X_{m_{t}+1}^{(1)}$. However, this can only happen if $n_{3} \geq N_{t+1}$, so $n \geq N_{t+1}$ in this case; this proves the claim.

Now let $m,n$ and $t$ be as in the claim, so that $m_{1} \leq m_{t} < m \leq m_{t+1}$ and $F_{n} = X_{m}^{(i)}$ for some $i \in \{1,2,\dots,p\}$. Since $m \leq m_{t+1}$, then we have 
\begin{equation*}
F_{n} \subset \bigcup_{k=1}^{m_{t+1}}X_{k}.
\end{equation*}

Hence by \eqref{gammaDefn2}, the definition of $N_{t+1}$ and the fact that $n \geq N_{t+1}$, it follows that
\begin{equation}\label{supPart2}
\sup\{|x|: x \in F_{n}\} \leq \gamma_{m_{t+1}}\leq a_{n}.
\end{equation}

Finally, since for all $n \geq N_{1}$ we have $F_{n} = X_{m}^{(i)}$ for some $m\geq m_{1}$ and $i \in \{1,2,\dots,p\}$, it follows from \eqref{supPart1} and \eqref{supPart2} that $|f^{n}(\zeta)|\leq a_{n}$ as required.
\end{proof}

To complete the proof of Theorem~\ref{MainTheorem} for mappings with infinitely many poles, we shall use a specific case of \cite[Lemma~5.1]{Warren1}. 

\begin{lem} \label{EExists}
	Let $f: \mathbb{R}^d \to \hat{\mathbb{R}}^d $ be a quasimeromorphic mapping of transcendental type. Suppose that there exists an open bounded neighbourhood $U \subset \mathbb{R}^d$ of a pole of $f$ such that $f^{-1}(u)$ is infinite for all $u \in \overline{U}$. Then given any $r>0$, there exists an open bounded region $E_{U} \subset A(r,\infty)$ such that $f(U) \supset \overline{E_{U}}$ and $f(E_{U}) \supset \overline{U}$.
\end{lem}

\begin{proof}[Proof of Theorem~\ref{MainTheorem}: Infinitely many poles]
	Let $f$ have a sequence of poles $(x_{m})$ tending to $\infty$. Now through Lemma~\ref{EExists} and choosing a subsequence of the poles and relabelling, we can construct the sequences $(R_{m})$, $(U_{m})$ and $(E_{m})$ by induction as follows. 
	
	Initialise $R_{1}=0$ and suppose that $R_{m}$ has been chosen for some $m \in \mathbb{N}$. By removing finitely many terms and relabelling, we may assume without loss of generality that $x_{m} \in A(R_{m},\infty)$ and $x_{m}$ is not an exceptional point. Now set $U_{m}$ to be an open bounded neighbourhood of $x_{m}$, such that $\overline{U_{m}} \subset A(R_{m}, \infty)$ and $f^{-1}(u)$ is infinite for all $u \in \overline{U_{m}}$. By applying Lemma~\ref{EExists}, choose a non-empty open bounded region $E_{m} \subset A(R_{m}, \infty)$ such that 
	\begin{equation} \label{waiting}
	f(U_{m}) \supset \overline{E_{m}} \text{, and } f(E_{m}) \supset \overline{U_{m}}.
	\end{equation}
Finally, choose $R_{m+1} \geq m+1$ such that $A(R_{m+1}, \infty) \subset f(U_{m})$. 

	With the $R_{m}$, $U_{m}$ and $E_{m}$ established, we shall now choose the sets $X_{m}^{(i)}$ that satisfy the hypotheses in Lemma~\ref{lem:HoldUp2} with $p =2$. For each $m \in \mathbb{N}$, define $X_{m}^{(1)} = U_{m}$ and $X_{m}^{(2)} = E_{m}$. Here, it should be noted that we are taking the subsequence $m_{t} =t$ for all $t \in \mathbb{N}$. Firstly, note that \eqref{SmallestXToInfty2} is satisfied, as $\inf\{|x| : x \in U_{m} \cup E_{m}\} \geq R_{m}$ and $R_{m} \to \infty$ as $m \to \infty$. 
	
	Now since every $U_{m}$ is an open neighbourhood of a pole, then $U_{m} \cap J(f) \neq \varnothing$. Also by \eqref{waiting} and Theorem~\ref{JuliaSetProperties}(iii), then $E_{m} \cap J(f) \neq \varnothing$ as well, so (X6) is satisfied. Further, (X5) is satisfied by \eqref{waiting} since for all $m \in \mathbb{N}$,
	\begin{equation*}
	f(X_{m}^{(1)}) \supset X_{m}^{(2)} \text{, and } f(X_{m}^{(2)}) \supset X_{m}^{(1)}.
	\end{equation*}
	
	To show (X4) is satisfied, observe that by construction,
	\begin{equation*}
	f(X_{m}^{(1)}) = f(U_{m}) \supset A(R_{m+1}, \infty) \supset U_{m+1} = X_{m+1}^{(1)}.
	\end{equation*}
	
Finally, an application of Lemma~\ref{lem:HoldUp2} completes the proof of Theorem~\ref{MainTheorem} for functions with an infinite number of poles.
	\end{proof}

\section{Proof of Theorem~\ref{BoundaryTheorem} and counterexamples}

\subsection{Sufficient conditions for Theorem~\ref{BoundaryTheorem}(i)}
Let $f$ be a $K$-quasimeromorphic mapping of transcendental type with at least one pole. To prove Theorem~\ref{BoundaryTheorem}(i), we shall provide sufficient conditions for the existence of infinitely many points in $BO(f) \cap J(f)$ and $BU(f) \cap J(f)$. Sets that satisfy these conditions will then be identified from each case of the proof of Theorem~\ref{MainTheorem}.

Firstly, suppose that there exists some non-empty bounded set $U_{0}$ with $\overline{U_{0}} \cap J(f) \neq \varnothing$ such that
 
\begin{enumerate}
\item[(BO1)] there exists some $N \in \mathbb{N} \cup \{0\}$ and bounded sets $U_{t}$ where $f(U_{N}) \supset U_{0}$ and if $N \geq 1$, then $f(U_{t}) \supset U_{t+1}$ for all $0 \leq t \leq N-1$.
\end{enumerate}

Then by applying Lemma~\ref{NestedSeqn} with $F_{n} = U_{[n]_{N+1}}$ for all $n \in \mathbb{N}$, we get that there exists some $x \in J(f) \cap BO(f) \cap \overline{U_{0}}$. By finding infinitely many such $U_{0}$ with pairwise disjoint closures, then we can conclude that $J(f) \cap BO(f)$ is infinite.

Next, let $V$ be a non-empty bounded set and let $(k_{t})$ be a sequence of natural numbers such that 
\begin{enumerate}
 \item[(BU1)] for each $t \in \mathbb{N}$, there exists a non-empty bounded set $V_{t}$ with $f^{k_{t}}(Y_{t}) \supset V_{t}$ for some subset $Y_{t} \subset V$, and $f^{m_{t}}(Z_{t}) \supset V$ for some subset $Z_{t} \subset V_{t}$ and some $m_{t} \in \mathbb{N}$; and
 \item[(BU2)] $\inf\left\lbrace |x| : x \in V_{t} \right\rbrace \to \infty$ as $t \to \infty$.
 \end{enumerate}
 
Then by applying Lemma~\ref{NestedSeqn} with $F_{2n-1} = Y_{n}$ and $F_{2n} = Z_{n}$ for all $n \in \mathbb{N}$, this gives a sufficient condition for the existence of a point $x \in \mathbb{R}^d$ in $BU(f)$. Moreover, if we have that

\begin{enumerate}
\item[(BU3)] $J(f) \cap \overline{Y_{t}} \neq \varnothing$ for all $t \in \mathbb{N}$,
\end{enumerate} 
then Lemma~\ref{NestedSeqn} gives us a point $y \in J(f) \cap BU(f) \cap \overline{V}$. Recalling Theorem~\ref{JuliaSetProperties}(iii), it is clear that $f^k(y) \in J(f) \cap BU(f)$ for all $k \in \mathbb{N}$, hence it follows that $J(f)\cap BU(f)$ is infinite.

\subsection{Proof of Theorem~\ref{BoundaryTheorem}(i)}

Let $f:\mathbb{R}^d \to \hat{\mathbb{R}}^d$ be a quasimeromorphic mapping of transcendental type with at least one pole but finitely many poles. As $f$ has finitely many poles then by taking $R>0$ sufficiently large, we have that $f: A(R, \infty) \to \mathbb{R}^d$ is a quasiregular mapping with an essential singularity at infinity. We shall first show that $BO(f) \cap J(f)$ and $BU(f) \cap J(f)$ are infinite when $f$ restricted to $A(R,\infty)$ has the pits effect. Indeed, by Lemma~\ref{pitsCoverN} and the arguments directly after Lemma~\ref{pitsMovement}, there exist bounded open balls $B_{t}$, $t \in \mathbb{N}$, such that 

\begin{enumerate}
\item[(i)] $f(B_{t}) \supset B_{s}$ for all $s \leq t$;

\item[(ii)] there exists some natural sequence $(b_{t})$ and some sets $Y_{t} \subset B_{1}$ such that $f^{b_{t}}(Y_{t}) \supset B_{t+1}$ for all $t \in \mathbb{N}$;

\item[(iii)] $\inf\{|x| : x \in B_{t}\} \to \infty$ as $t \to \infty$;

\item[(iv)] $\overline{B_{t}}$ are all pairwise disjoint; and

\item[(v)] $\overline{B_{t}} \cap J(f) \neq \varnothing$ for all $t \in \mathbb{N}$.
 \end{enumerate}

(BO1) is clearly satisfied from (i) and (v), by setting $N=0$ and $U_{0}=B_{1}$. It then follows from (iv) that this can be repeated for each set $B_{t}, t \in \mathbb{N}$ to get infinitely many points. Therefore $BO(f) \cap J(f)$ is infinite.

Now set $V = B_{1}$ and $V_{t} = B_{t+1}$. Then (BU1) is satisfied by (i) and (ii), with $m_{t}=1$ and $k_{t} = b_{t}$ for all $t \in \mathbb{N}$. In addition, (BU2) is satisfied by (iii) and (BU3) is satisfied by (v) and the backward invariance of $J(f)$. Therefore $BU(f) \cap J(f)$ is infinite.

For the other cases, we can follow a similar argument. Indeed, suppose that $f$ is a quasimeromorphic mapping of transcendental type with at least one pole but finitely many, whose restriction to $A(R,\infty)$ for some $R>0$ is a quasiregular mapping that does not have the pits effect. Then from Lemma~\ref{holdUpNoPits}, the arguments immediately after Lemma~\ref{holdUpNoPits}, and Lemma~\ref{BU(f)(ii)}, there are non-empty bounded sets $T_{t}W_{j}$ with $t \in \mathbb{N}$ and $j \in \{1,2,\dots, q_{0}\}$, such that

\begin{enumerate}
\item[(i)] for each $t \in \mathbb{N}$ and $j \in \{1,2,\dots,q_{0}\}$, there exists $i \in \{1,2,\dots,q_{0}\}$ such that $f(T_{t}W_{j}) \supset T_{t}W_{i}$;

\item[(ii)] there exists some constants $i_{0}, j_{0} \in \{1,2,\dots,q_{0}\}$ and $d \in \mathbb{N}$ such that for each $t \in \mathbb{N}$, there is some $c_{t} \in \mathbb{N}$, some subset $Y_{t} \subset T_{2}W_{j_{0}}$ and some subset $Z_{t} \subset T_{t+2}W_{i_{0}}$ where

\begin{equation*}
f^{c_{t}}(Y_{t}) \supset T_{t+2}W_{i_{0}} \text{ and } f^{d}(Z_{t}) \supset T_{2}W_{j_{0}}.
\end{equation*}

\item[(iii)] $\inf\{|x| : x \in \bigcup_{j}T_{t}W_{j}\}\to \infty$ as $t \to \infty$;

\item[(iv)] $T_{t}\overline{W_{j}}$ are all pairwise disjoint; and

\item[(v)] $T_{t}\overline{W_{j}} \cap J(f) \neq \varnothing$ for all $t \in \mathbb{N}$ and $j \in \{1,2,\dots,q_{0}\}$.
\end{enumerate}

Now fix some $t \in \mathbb{N}$ and set $S_{0} = T_{t}W_{1}$. Then using (i), for each $k \in \mathbb{N}$ set $S_{k} = T_{t}W_{i}$ where $T_{t}W_{i} \subset f(S_{k-1})$. Note by (v) that $\overline{S}_{k} \cap J(f) \neq \varnothing$ for all $k \in \mathbb{N}$. 

As $S_{k} \in \{T_{t}W_{i} : i \in \{1,2,\dots,q_{0}\}\}$ for each $k \in \mathbb{N}$, then there must exist some numbers $n_{1}, n_{2} \in \mathbb{N}$, with $n_{1} < n_{2}$, such that $S_{n_{2}}=S_{n_{1}}$. This means that $f(S_{n_{2}-1}) \supset S_{n_{2}}=S_{n_{1}}$. Hence (BO1) is satisfied with $N=n_{2}-n_{1}-1$ and $U_{t} = S_{n_{1}+t}$. It then follows from (iv) and (v) that $BO(f) \cap J(f)$ is infinite.

Next, using (ii) set $V=T_{2}W_{j_{0}}$ and for each $t \in \mathbb{N}$ set $V_{t}= T_{t+2}W_{i_{0}}$. Now it follows from (ii) that $f^{d}(Z_{t}) \supset T_{2}W_{j_{0}}$ for some constant $d \in \mathbb{N}$ and some subsets $Z_{t} \subset V_{t}$. Also, there exists some sequence $c_{t}$ such that $f^{c_{t}}(Y_{t}) \supset V_{t}$ for some subsets $Y_{t} \subset V$. This means that (BU1) is satisfied with $k_{t}=c_{t}$ and $m_{t}=d$ for each $t \in \mathbb{N}$. Further, (BU2) is given by (iii) whilst (BU3) follows from (v) and the backward invariance of $J(f)$. Hence $BU(f) \cap J(f)$ is infinite in this case.

When $f$ has infinitely many poles, for $t,m \in \mathbb{N}$ we can choose neighbourhoods of poles $D_{t}$ and use Lemma~\ref{EExists} to get non-empty bounded sets $E_{t,m}$ such that
\begin{enumerate}
\item[(i)] for each fixed $t \in \mathbb{N}$, we have $f(D_{t}) \supset E_{t,m}$ and $f(E_{t,m}) \supset D_{t}$ for all $m \in \mathbb{N}$;

\item[(ii)] for each fixed $t \in \mathbb{N}$, then $\inf\{|x|: x \in E_{t,m}\} \to \infty \text{ as } m \to \infty$;

\item[(iii)] $\overline{D_{t}}$ are all pairwise disjoint and $\overline{E_{t,m}}$ are all pairwise disjoint; and

\item[(iv)] for each $t, m \in \mathbb{N}$ we have $\overline{D_{t}} \cap J(f) \neq \varnothing$ and $\overline{E_{t,m}} \cap J(f) \neq \varnothing$.
\end{enumerate}

For each $t \in \mathbb{N}$, setting $U_{0} = D_{t}$ and $U_{1} = E_{t,1}$ satisfies (BO1) by (i). It then follows by (iii) that $BO(f) \cap J(f)$ is infinite.

Further, set $V = D_{1}$ and $V_{m} = E_{1,m}$. Then (BU1)~-~(BU3) are all given by (i), (ii) and (iv) respectively, hence $BU(f) \cap J(f)$ is infinite; this completes the proof of Theorem~\ref{BoundaryTheorem}(i).

\subsection{Proof of Theorem~\ref{BoundaryTheorem}(ii)}

Theorem~\ref{BoundaryTheorem}(ii) shall be attained as a corollary to the following result, which is similar to \cite[Lemma~10]{RS5}.

\begin{lem}
Let $f: \mathbb{R}^d \to \hat{\mathbb{R}}^d$ be a quasimeromorphic mapping of transcendental type with at least one pole. Suppose that there is an infinite set $X \subset \mathbb{R}^d$ such that $X$ is completely invariant under $f$ and $\mathbb{R}^d \setminus (X \cup \mathcal{O}_{f}^{-}(\infty))$ is infinite. Then $J(f) \subset \partial X$.
\end{lem}

\begin{proof}
Let $x \in J(f)$ and let $U_{x}$ be an arbitrary neighbourhood of $x$. Since $X$ and $\mathbb{R}^d \setminus (X \cup \mathcal{O}_{f}^{-}(\infty))$ are infinite sets, then $X \setminus E(f)$ and $\mathbb{R}^d \setminus (X \cup \mathcal{O}_{f}^{-}(\infty) \cup E(f))$ are non-empty. Now $X$ and $\mathbb{R}^d \setminus (X \cup \mathcal{O}_{f}^{-}(\infty))$ are both completely invariant, so by Theorem~\ref{JuliaSetProperties}(vi) it follows that $X \cap U_{x} \neq \varnothing$ and $\left(\mathbb{R}^d \setminus (X \cup \mathcal{O}_{f}^{-}(\infty))\right) \cap U_{x} \neq \varnothing$.

As $X$ and $\mathbb{R}^d \setminus (X \cup \mathcal{O}_{f}^{-}(\infty))$ are disjoint, then we must have $\partial X \cap U_{x} \neq \varnothing$. Finally, since $U_{x}$ was arbitrary, then $x \in \partial X$ as required.
\end{proof}

Since $I(f), BO(f)$ and $BU(f)$ are all completely invariant and disjoint, then the result follows from Theorem~\ref{BoundaryTheorem}(i).

\subsection{Counterexamples}

To show that the reverse inclusion in Theorem~\ref{BoundaryTheorem} does not necessarily hold, we shall first construct a mapping similar to those found in \cite[Example 7.3]{BN1} and \cite[Example 1]{Dominguez}. This will give a mapping $f$ such that $\left( \partial I(f) \cap \partial BO(f) \right) \setminus J(f) \neq \varnothing$.

\begin{ex}\label{Example1}
Let $h: \mathbb{C} \to \hat{\mathbb{C}}$ be the transcendental meromorphic function defined by $h(z)= 2 + \exp(-z) + (z+1)^{-1}$, and define $g: \mathbb{C} \to \hat{\mathbb{C}}$ by $g(z) = z + h(z)$. Firstly, note that if $z$ is in the half-plane $H_{1}:=\{z: \real(z) > 1\}$, then $h(z) \in \{v: 1<\real(v)<3\}$. Now, we have $g(H_{1}) \subset H_{1}$ and $g^n(z) \to \infty$ as $n \to \infty$ whenever $z \in H_{1}$.

Next, for a large constant $M \in \mathbb{R}$ define $f:\mathbb{C} \to \hat{\mathbb{C}}$ by
\[
f(z) = \left\lbrace \begin{array}{ll}
g(z) & \text{if } \real(z)<M \text{ or } \real(z)>2M, \\
g(z) + h(z)\sin\left( \frac{\pi \real(z)}{M} \right)  & \text{if } M\leq \real(z) \leq 2M.
	\end{array} \right.
\]

It is easy to see that $f$ is a quasimeromorphic mapping of transcendental type with one pole if $M$ is large. 

Similar to $g$, we have that $f(H_{1}) \subset H_{1}$, so $H_{1} \cap J(f) = \varnothing$. Also, the point $w=3M/2$ is such that $f(w)=w$, while $f(x)>x$ for all real $x>w$. This means that $f^n(x) \to \infty$ as $n \to \infty$ for all real $x >w$, thus $(w,\infty) \subset I(f)$ and $w \in BO(f)$. Therefore $w \in (\partial I(f) \cap \partial BO(f)) \setminus J(f)$.

This example can be extended to a quasimeromorphic mapping of transcendental type $\tilde{f}:\mathbb{C} \to \hat{\mathbb{C}}$ with infinitely many poles, by replacing $h$ with $\tilde{h}:\mathbb{C} \to \hat{\mathbb{C}}$ defined by 
\begin{equation*}
\tilde{h}(z) = 2 + \exp(-z) + \sum_{k=1}^{\infty}(z+2^k -1)^{-1},
\end{equation*}
and replacing $g$ with $\tilde{g}: \mathbb{C} \to \hat{\mathbb{C}}$ defined by $\tilde{g}(z) = z+ \tilde{h}(z)$. Here, since $|z+2^k-1|>2^k$ for all $z \in H_{1}$, then $\tilde{h}(z) \in \{v : 1/2 < \real(v)<7/2 \}$ on $H_{1}$. This means that the behaviour of $H_{1}$ and $w=3M/2$ under $\tilde{g}$, hence also for $\tilde{f}$, remains the same as above. 
\end{ex}

The final example is a direct modification of the example constructed in \cite{NS3}, as we will only require specific dynamics in the upper half plane to find a point in $\partial BU(f) \setminus J(f)$.

\begin{ex}\label{Example2}
Let $h:\mathbb{C} \to \mathbb{C}$ be the quasiconformal mapping constructed in \cite[Proof of Theorem~4]{NS3}. This map is such that $BU(h)$ and $BO(h)$ intersect the upper half plane $\mathbb{H} := \{z : \im(z) > 0\}$ non-trivially, and $h(\mathbb{H}) \subset \mathbb{H}$.

Next for a small constant $\alpha>0$, let $g: \mathbb{C} \to \hat{\mathbb{C}}$ be defined by
\[
g(z) = \left\lbrace \begin{array}{ll}
z & \text{if } \im(z) \geq 0, \\
z -\alpha(\im(z))\left( \exp(-z^2) + (z + 4i)^{-1} \right)  & \text{if } \im(z) \in [-1,0), \\
z +\alpha\left( \exp(-z^2) + (z + 4i)^{-1} \right) & \text{otherwise}, \\
	\end{array} \right.
\]

It is easy to show that if $\alpha$ is sufficiently small, then $g$ is a quasimeromorphic mapping of transcendental type with one pole. Note that $g$ is the identity mapping on the upper half plane, so $g(\mathbb{H}) \subset \mathbb{H}$. 

Now the mapping $f := g \circ h$ is also a quasimeromorphic mapping of transcendental type with one pole. It follows that $f(\mathbb{H}) \subset \mathbb{H}$ and so $J(f) \cap \mathbb{H} = \varnothing$. Further since $g$ is the identity mapping on $\mathbb{H}$, then $f$ has the same dynamics on $\mathbb{H}$ as $h$. This means that $\mathbb{H} \cap BU(f) \neq \varnothing$ and $\mathbb{H} \cap BO(f) \neq \varnothing$. As $BO(f)$ and $BU(f)$ are disjoint, then $\mathbb{H} \cap \partial BU(f) \neq \varnothing$, hence $\partial BU(f) \setminus J(f) \neq \varnothing$ as required.

By making a simple modification, we can also create a quasimeromorphic mapping of transcendental type with infinitely many poles, by replacing $(z + 4i)^{-1}$ in the definition of $g(z)$ with $\sum_{k=1}^{\infty}(z + 2^k + 4i)^{-1}$; the dynamics of the new function remain unchanged in $\mathbb{H}$ and hence the result follows.
\end{ex}

\end{document}